\documentclass[11pt,a4paper]{article}
\usepackage{amsfonts}
\usepackage{makecell}
\usepackage{graphics}
\usepackage{subfigure}
\usepackage{indentfirst}
\usepackage{amsmath}
\usepackage{amsmath,amssymb,amsthm}
\usepackage{amssymb}
\usepackage{algorithm,algpseudocode}
\usepackage{latexsym}
\usepackage{natbib}
\usepackage[margin=1in,a4paper]{geometry}
\usepackage{colortbl}
\usepackage{color}
\usepackage[colorlinks,citecolor=blue,urlcolor=blue]{hyperref}
\usepackage{mathtools}
\usepackage{setspace}
\usepackage{amsmath, amsthm, amssymb, bm, graphicx, hyperref, mathrsfs}
\usepackage{verbatim}
\usepackage{authblk}
\usepackage{epstopdf}
\usepackage{tikz}

\DeclareMathOperator{\tr}{tr}

\mathtoolsset{showonlyrefs=true}

\newtheorem{lemma}{Lemma}

\newtheorem{theorem}{Theorem}
\theoremstyle{remark}
\newtheorem{remark}{Remark}

\allowdisplaybreaks[4]

\graphicspath{{../../Codes/}}

\newenvironment{keywords}[1]{\par\noindent\textbf{Keywords: }#1}

\title{Linear, decoupled, positivity preserving, positive-definiteness preserving and energy stable schemes for the diffusive Oldroyd-B coupled with PNP model}

\author[1]{Wenxing Zhu}
\author[2]{Mingyang Pan}
\author[1,*]{Dongdong He}

\affil[1]{School of Science and Engineering, The Chinese University of Hong Kong, Shenzhen, Shenzhen, Guangdong, 518172, P.R.China}
\affil[2]{School of Science, Hebei University of Technology, Tianjin, 300401, P.R.China}
\affil[*]{Corresponding author: \texttt{hedongdong@cuhk.edu.cn}}

\begin{document}

\maketitle




\begin{abstract}
In this paper, we present a first-order finite element scheme for the viscoelastic electrohydrodynamic model. The model incorporates the Poisson-Nernst-Planck equations to describe the transport of ions and the Oldroyd-B constitutive model to capture the behavior of viscoelastic fluids. To preserve the positive-definiteness of the conformation tensor and the positivity of ion concentrations, we employ both logarithmic transformations. The decoupled scheme is achieved by introducing a nonlocal auxiliary variable and using the splitting technique.
The proposed schemes are rigorously proven to be mass conservative and energy stable at the fully discrete level. To validate the theoretical analysis, we present numerical examples that demonstrate the convergence rates and the robust performance of the schemes. The results confirm that the proposed methods accurately handle the high Weissenberg number problem (HWNP) at moderately high Weissenberg numbers. Finally, the flow structure influenced by the elastic effect within the electro-convection phenomena has been studied.
\end{abstract}

\keywords{viscoelastic fluids, fully-decoupled scheme, characteristic finite element, energy stable, logarithmic transformation}

\section{Introduction}
Electrohydrodynamic flow concerns the motions of ionized particles or molecules, the dynamics of electrically charged fluids, and their interactions with electric fields and the surrounding fluids. Physically, the motion of fluid flow is driven by the Coulomb force created by the ions under the electric field and the ionic diffusion is driven by the concentration gradients of the ions. Due to the presence of ions, the resulted electrical field can affect the behavior and distribution of charged ions in the fluid. Such phenomena have been popular in a wide range of industrial and commercial areas such as electrophoretic separation of macromolecules~\cite{ghosal2006electrokinetic}, inkjet printing~\cite{paul2022numerical,guan2022numerical,jiang2021cfd}, heat transfer enhancement~\cite{grassi2005ehd,testi2018heat,laohalertdecha2007review}.

Viscoelastic fluid is a type of complex fluid which has both viscous and elastic properties, as encountered in cosmetics industry~\cite{gallegos1999rheology}, food processing industries~\cite{bistany1983comparison} and blood~\cite{montecinos2014hyperbolic}. Compared with the Newtonian fluid, the total stress for the viscoelastic flow has both viscous stress and elastic stress while Newtonian fluids only have viscous
stress. The dynamics of viscoelastic flow is governed by the conservation of mass and momentum equations with different constitutive equations. Typical constitutive laws are Oldroyd-B~\cite{oldroyd1950formulation}, Giesekus~\cite{giesekus1982simple}, Phan-Thien-Tanner (PTT)~\cite{thien1977new}, eXtended Pom-Pom (XPP)~\cite{verbeeten2001differential}, among which the Oldroyd-B model is the simplest one. The simulation of these flow model is a challenging task from the both theoretical and numerical points of view. The constitutive equation is highly coupled and advection dominated which may induce both global and local oscillations in the numerical solution. In addition, all numerical schemes to simulate viscoelastic flows meet the so-called HWNP, i.e. the difficulty of convergence of numerical algorithms encountered when the Weissenberg number is above certain values. A mechanism responsible for instability seen at high Weissenberg number has been proposed in ~\cite{boyaval2009free,fattal2005time}. The failure to properly balance the the exponential growth of deformation with the convection can cause a numerical instability. To improve the numerical stability, many approaches are developed. ~\cite{hulsen1990sufficient} rewrited the stress equation in terms of conformation stress tensor and showed that conformation stress tensor is symmetric positive-definite. The failure to satisfy positivity can yield a numerical instability throughout the computations. To preserve the positive definiteness of the tensor, various mathematical reformulations and numerical schemes have been developed. ~\cite{lee2006new} used the direct discretization of the objective derivative. The square-root of the conformation tensor was introduced in~\cite{balci2011symmetric,lozinski2003energy}. The logarithm conformation representation (LCR) method was proposed by ~\cite{fattal2004constitutive,hulsen2005flow}. ~\cite{knechtges2014fully} proposed a new but fully implicit LCR method to avoid the eigen-decomposition of the velocity gradient. In addition, the kernel conformation transformation was applied by ~\cite{afonso2012kernel}.


As we all know, for the constituents
of the full electrohydrodynamic model of Newtonian fluid, there are many effective methods available for each individual equation ~\cite{pan2021energy,pan2020unconditionally,he2019positivity}.
For the separate subproblems comprising the Oldroyd-B electrohydrodynamic problem, there exists many efficient numerical methods. So far, for the Oldroyd-B viscoelastic fluid model, ~\cite{lukavcova2016energy} proposed energy dissipative characteristic schemes for the diffusive
Oldroyd-B viscoelastic fluid.~\cite{varchanis2019new} presented proposed a new, fully consistent and highly stable formulation by combination of classical finite element stabilization techniques with LCR of the constitutive equation. As far as the author knows, these numerical schemes are nonlinear and coupled while maintaining energy stability.~\cite{zhao2015simpler} proposed a simpler GMRES method combined with finite volume method for simulating viscoelastic flows.~\cite{castillo2015first} presented first, second and third order fractional step methods for the three-field viscoelastic flow.~\cite{venkatesan2017three} presented a three-field local projection stabilized formulation. For the Poisson-Nernst-Planck problem to describe the dynamics of ions under an electric field, various schemes have been applied in  \cite{gao2017linearized,prohl2009convergent,shen2021unconditionally,liu2022positivity,shen2020decoupling}.
~\cite{pan2021positive} presented positive-definiteness preserving and energy stable time-marching scheme for a diffusive Oldroyd-B electrohydrodynamic model. Our research in this paper can be regarded as an extension of ~\cite{pan2021positive}.



Motivated by the nonlocal auxiliary variable method developed in~\cite{he2023decoupled, zhou2023efficient, wang2023decoupled, pan2024linear, yang2021novel, shen2018scalar}, we aim to present a linear, decoupled, conservative, positivity-preserving and energy stable scheme for the viscoelastic Oldroyd-B flow.
The main challenge is how to deal with the nonlinear coupling terms in the fully discrete scheme while maintaining the energy stability. Thus, to overcome it, we introduce a nonlocal variable and an ordinary differential equation(ODE) associated with it.
The ODE can allow to construct a linear and explicit scheme for discretizing the nonlinear and coupling terms.
In this work, the backward Euler scheme coupled with the projection method of the Navier-Stokes equations and characteristic finite element method of log-conformation tensor is designed.
The resulting scheme satisfies the following properties: (1) discrete energy stable; (2) mass conservative;
(3) preserving positivity of concentrations; (4) preserving the positive-definiteness of conformation tensor;
(5) the numerical scheme can be implemented by solving decoupling linear equations.
To the best of the author's knowledge, the scheme developed in this article is the first to have above five characteristics.

The rest of this paper is organized as follows. In Section 2, we reformulate the mathematical model and formally derive the free energy dissipation law for viscoelastic electrohydrodynamic model.
In Section 3, we reformulate the model based on auxiliary variable approach and derive the free energy dissipation law.
In Section 4, we design a linear and energy stable scheme. 
A fully decoupled numerical scheme is constructed in section 5 and we further describe its implementations in detail. In Section 6, numerical results are presented to validate our schemes. Conclusions are given in the last section.

\section{Mathematical model}

\subsection{Governing equations of viscoelastic Oldroyd-B model}
In this work, we consider the incompressible viscoelastic fluid model in a bounded domain $\Omega\subset \mathbb{R}^2$.  Additionally, gravity effects is neglected. 
the induced magnetic fields are usually neglected due to very low currents in the liquids.

The model follows the incompressible Navier-Stokes equations,
the viscoelastic stress in the fluid can be described by
\begin{subequations}\label{model}
\begin{align}
&\partial_tc_i + \mathbf{u}\cdot\nabla c_i = D\nabla\cdot(c_i\nabla g_i ),\label{concentration} \\
&g_i = \log \frac{c_i}{c_0} + \frac{z_i e}{k_BT} V,\\
&-\nabla\cdot(\epsilon\nabla V) = \sum\limits_{i}z_ic_i,\label{Possion}\\
&\nabla\cdot \mathbf{u}=0,\label{incompressible} \\
&\rho_f(\partial_t\mathbf{u} + (\mathbf{u}\cdot\nabla)\mathbf{u})  + \nabla p = \nabla\cdot\mathbf{T} - \sum\limits_{i}z_ic_i\nabla V,\label{NS} \\
&\mathbf{T} +\lambda_1 \stackrel{\triangledown}{\mathbf{T}} =2\mu(\mathbb{D}+\lambda_2\stackrel{\triangledown}{\mathbb{D}}) ,\label{Oldroyd-B}
\end{align}
\end{subequations}
where $c_i$ is the ion concentration of the ith species with $i\in 1,...,N$, $g_i$ is the corresponding chemical potentials, $z_i$ is the ionic valency, $D$ is the diffusion constant, $k_B$ is the Boltzmann's constant, $T$ is the absolute temperature, $e$ is the unit charge, $V$ is the electric potential, $\epsilon$ is the electric permittivity, $\rho_f$ is the fluid density, $\mathbf{u}=(u,v)$ is the fluid velocity, $p$ is the pressure, $\mathbb{D}=\frac{1}{2}[\nabla \mathbf{u} + (\nabla \mathbf{u})^T]$ is the deformation tensor, $\mathbf{T}$ is the total stress tensor, $\mu$ is the total viscosity, $\lambda_1$, $\lambda_2$ are the relaxation time and retardation time, respectively. The relaxation time $\lambda_1$ is assumed to be bigger than the retardation time $\lambda_2$. The upper-convected derivative is defined by
\begin{align*}
\stackrel{\triangledown}{\bm{\tau}} = \partial_t\bm{\tau} + (\mathbf{u}\cdot\nabla)\bm{\tau} -\nabla\mathbf{u}\cdot\bm{\tau}- \bm{\tau}\cdot \nabla\mathbf{u}^{T}.
\end{align*}
The stress tensor $\mathbf{T}$ consists of the purely viscous component $2\mu\alpha \mathbb{D}$ with $ \alpha=\frac{\lambda_2}{\lambda_1}$ and the elastic component $\bm{\tau}$, namely
\begin{align*}
\mathbf{T} = \bm{\tau} + 2\mu\alpha\mathbb{D}.
\end{align*}
By replacing $\mathbf{T}$ in \eqref{Oldroyd-B} with $\bm{\tau}$ and and using the fact that $\stackrel{\triangledown}{\mathbf{I}}=-2\mathbb{D}$, the Oldroyd-B constitutive equation with elastic
stress tensor $\bm{\tau}$ is given by
\begin{align}\label{Oldroyd-B-elastic-part}
\bm{\tau} +\lambda_1 \stackrel{\triangledown}{\bm{\tau}} =2\mu_p\mathbb{D},
\end{align}
where $\mu_p=\mu (1-\alpha)$ is polymer viscosity. Now we introduce the dimensionless conformation tensor $\bm{\sigma}$ as
\begin{align}\label{sigma}
\bm{\sigma}= \mathbf{I} + \frac{\lambda_1}{\mu (1-\alpha)}\bm{\tau},
\end{align}
which has the positive definite property \cite{boyaval2009free,hulsen1990sufficient}. Substituting \eqref{sigma} into \eqref{Oldroyd-B-elastic-part} , the model \eqref{model} can be rewritten as
\begin{subequations}\label{New-model}
\begin{align}
&\partial_tc_i + \mathbf{u}\cdot\nabla c_i = D\nabla\cdot(c_i\nabla g_i ),\label{New-concentration} \\
&g_i = \log \frac{c_i}{c_0} + \frac{z_i e}{k_BT} V,\\
&-\nabla\cdot(\epsilon\nabla V) = \sum\limits_{i}z_ic_i,\label{New-Possion}\\
&\rho_f(\partial_t\mathbf{u} + (\mathbf{u}\cdot\nabla)\mathbf{u})  + \nabla p = \mu_s\Delta \mathbf{u}+ \frac{\mu_p}{\lambda_1}\nabla\cdot\bm{\sigma} - \sum\limits_{i}z_ic_i\nabla V,\label{New-NS} \\
&\nabla\cdot \mathbf{u}=0,\label{New-incompressible} \\
&\bm{\sigma} +\lambda_1 \stackrel{\triangledown}{\bm{\sigma}} = \mathbf{I} ,\label{New-Oldroyd-B}
\end{align}
\end{subequations}
where $\mu_s=\alpha\mu$ is solvent viscosity.
\subsection{Non-dimensionalisation}
To get a nondimensional formulation, we introduce the following dimensionless variables:
\begin{align*}
&\tilde{\mathbf{x}} = \frac{\mathbf{x}}{\hat{l}}, \quad \tilde{\mathbf{u}} = \frac{\mathbf{u}}{\hat{u}},\quad \tilde{t} = \frac{t}{\hat{l}/\hat{{u}}},\quad \tilde{c}_i = \frac{c_i}{c_0},\quad\tilde{V} =\frac{V}{k_BT/e}, \\
&\tilde{p} = \frac{p}{\rho_f \hat{{u}}^2},\quad \tilde{\bm{\tau}}=\frac{\bm{\tau}}{\mu_p/\lambda_1}.
\end{align*}
For clarity, we omit the superscript of the dimensionless variables. Taking into account the diffusive effects in the evolution equation of the elastic stress, the govern equations of the dimensionless diffusive viscoelastic electrohydrodynamic model with an Oldroyd-B constitutive equation \eqref{New-model} become:
\begin{subequations}\label{non-model}
\begin{align}
&\partial_tc_i +\nabla\cdot(\mathbf{u} c_i) = \frac{1}{Pe}\nabla\cdot(\nabla c_i + z_ic_i\nabla V ),\label{non-concentration} \\
&-\lambda\Delta V = \sum\limits_{i}z_ic_i,\label{non-Possion}\\
&\nabla\cdot \mathbf{u}=0,\label{non-incompressible} \\
&\partial_t\mathbf{u} + (\mathbf{u}\cdot\nabla)\mathbf{u}  - \frac{1}{Re} \Delta\mathbf{u} + \nabla p = M\nabla\cdot \bm{\sigma} -Co\sum\limits_{i}z_ic_i\nabla V,\label{non-NS} \\
&\partial_t\bm{\sigma} + (\mathbf{u}\cdot\nabla)\bm{\sigma} -\nabla\mathbf{u}\cdot\bm{\sigma}- \bm{\sigma}\cdot \nabla\mathbf{u}^{T} =  \frac{1}{Wi} (\mathbf{I}-\bm{\sigma}) + \kappa_1 \Delta \bm{\sigma},\label{non-Oldroyd-B}
\end{align}
\end{subequations}
where $\kappa_1>0$ is a diffusive parameter and the nondimensional numbers are defined as follows
\begin{align*}
Re=\frac{ \rho_f\hat{u} \hat{l}}{\mu_s},\quad Co=\frac{c_0k_BT}{\rho_f \hat{u}^2e},\quad Pe=\frac{\hat{l}\hat{u}}{D},\quad \lambda=\frac{\epsilon k_BT}{\hat{l}^2c_0e},\quad Wi=\frac{\lambda_1\hat{u}}{\hat{l}},\quad M=\frac{\mu_p}{\rho_f \hat{u}^2\lambda_1}.
\end{align*}
Here, $Re$ is the Reynolds number, $Co$ is the Coulomb-driven number, $Pe$ is the P\'{e}lect number, $\lambda$ is the ratio of Debye length to the characteristic length, $Wi$ is the Weissenberg number, and $M$ is the Ratio of elasticity to inertia.
The initial and boundary conditions are given by
\begin{align}
&c_i|_{t=0}=c_{i0}, \quad V|_{t=0}=V_0, \quad \bm{u}|_{t=0}=\bm{u}_0, \quad \bm{\sigma}|_{t=0} = \bm{\sigma}_0, \label{eq:IC}\\
&\frac{\partial V}{\partial \mathbf{n}}|_{\partial\Omega}=0, \quad \frac{\partial c_i}{\partial \mathbf{n}}|_{\partial\Omega}=0, \quad \bm{u}|_{\partial\Omega}=\mathbf{0}, \quad \frac{\partial \bm{\sigma}}{\partial \mathbf{n}}|_{\partial\Omega}=\mathbf{0}, \label{eq:BC}
\end{align}
where $\mathbf{n}$ is the unit outward normal on the boundary $\partial\Omega$.

\subsection{Energy decay}
Firstly, some basic notations are presented to be used in subsequent presentations. 
We introduce the following functional spaces:
\begin{align*}
&L_{0}^2(\Omega)=\left\{q\in L^2(\Omega) : \int_{\Omega}qdx=0\right\},\\
&H_{0}^1(\Omega)=\left\{s\in H^1(\Omega) : s=0 \ \ \mbox{on}\ \ \partial\Omega\right\};\\
&\bm{X}=\left\{\bm{v}\in H_0^1(\Omega)^2: \bm{v}|_{\partial \Omega}=\mathbf{0}\right\},\\
&M=L_{0}^2(\Omega),\quad S=\left\{s\in H^1(\Omega):\int_{\Omega}sdx=0\right\},\\
&Q=\left\{\varphi\in H^1(\Omega)\right\},\\
&\bm{V}=\left\{\bm{\psi}=[\psi_{ij}], 1\leq i,j\leq 2, \psi_{12}=\psi_{21}, \psi_{ij}\in H^1(\Omega) \right\}.
\end{align*}
$L^2(\Omega)$ denotes the standard Lebesgue functional space which is equipped with the inner product $(f,g)= \int_{\Omega} f(\bm{x})g(\bm{x}) d\bm{x}$ and the $L^2$-norm $\|f\|_{L^2}=(f,f)^{\frac{1}{2}}$. The space $\bm{X}$ is equipped with their usual scalar product $(\nabla\bm{u},\nabla\bm{v})$ and norm $\|\nabla \bm{u}\|_0$.
The double contraction $\bm{\tau}:\bm{\sigma}$ between rank-two tensors $\bm{\sigma}$, $\bm{\tau}\in \mathbb{R}^{d\times d}$ is defined by:
\begin{align}\label{double-contraction}
\bm{\tau}:\bm{\sigma}=\tr(\bm{\tau}\bm{\sigma}^{T})=\tr(\bm{\tau}^{T}\bm{\sigma}) = \sum\limits_{1\leq i,j\leq d}\bm{\tau}_{ij}\bm{\sigma}_{ij}.
\end{align}
Notice that if $\bm{\tau}$ is anti-symmetric and $\bm{\sigma}$ is symmetric, then $\bm{\tau}:\bm{\sigma} = 0$. Next, some properties of the positive-definite matrix are given in the following lemmas and the proof can be
found in~\cite{boyaval2009free}.
\begin{lemma}
Let $\bm{\sigma}$, $\bm{\tau}\in \mathbb{R}^{d\times d}$ be two positive definite matrices, then it holds
\begin{align}
&\bm{\sigma}-\ln\bm{\sigma}-\bm{I} \ \text{is positive semi-definite and} \ \tr(\bm{\sigma}-\ln\bm{\sigma}-\mathbf{I})\geq 0  ,\label{lemma1.1}\\
&\bm{\sigma}+\bm{\sigma}^{-1}-2\mathbf{I} \ \text{is positive semi-definite and} \ \tr(\bm{\sigma}+\bm{\sigma}^{-1}-2\mathbf{I})\geq 0  ,\label{lemma1.2}\\
&\tr((\ln\bm{\sigma}-\ln\bm{\tau})\bm{\sigma})\geq\tr(\bm{\sigma}-\bm{\tau}),\label{lemma1.3}\\
&\nabla (\ln\bm{\sigma}):\nabla \bm{\sigma}\geq 0,\label{lemma1.4}\\
&\nabla \bm{\sigma}:\nabla \bm{\sigma}^{-1}\leq 0.\label{lemma1.5}
\end{align}
\end{lemma}

\begin{lemma}
For any positive definite matrix $\bm{\sigma}(t)\in( C^1([0,T)) )^{\frac{d(d+1)}{2}}$, we have for any $t\in[0, T )$ that
\begin{align}
&(\frac{d}{dt}\ln \bm{\sigma}):\bm{\sigma} = \tr(\bm{\sigma} \frac{d}{dt}\ln\bm{\sigma}) = \frac{d}{dt}\tr\bm{\sigma},\label{lemma2.1}\\
&(\frac{d}{dt}\bm{\sigma}):\bm{\sigma}^{-1} = \tr(\bm{\sigma}^{-1} \frac{d}{dt}\bm{\sigma}) = \frac{d}{dt}\tr(\ln\bm{\sigma}).\label{lemma2.2}
\end{align}
\end{lemma}

\begin{theorem}
Assume that the system \eqref{non-model} is supplied with boundary conditions \eqref{eq:BC} and symmetric positive definite initial condition $\bm{\sigma}_0$. The free energy satisfies:
\begin{align}\label{continuous-free-energy}
\frac{d}{dt} \int_{\Omega} \left(\frac{1}{2}|\mathbf{u}|^2 + Co\sum\limits_{i}c_i(\log c_i-1) + \frac{Co}{2}\lambda|\nabla V|^2 + \frac{M}{2}\tr(\bm{\sigma}-\ln\bm{\sigma}-\mathbf{I})\right) d\mathbf{x} &\leq 0.
\end{align}
\end{theorem}
\begin{proof}
Differentiating \eqref{non-Possion} with respect to time, and taking the $L^2$ inner product with $CoV$, we obtain
\begin{align}\label{non-Possion-1}
\frac{d}{dt}\int_{\Omega} \frac{Co}{2}\lambda|\nabla V|^2 d\mathbf{x}= 
-Co\Big((\nabla\cdot(\lambda\nabla V))_t,V\Big) = Co(\sum\limits_{i}z_i\partial_tc_i,V).
\end{align}
Taking the $L^2$ inner product of \eqref{non-concentration} with $Co(\log c_i + z_i V)$ and taking the summation for $i$ to get
\begin{align}\label{non-concentration-1}
&\frac{d}{dt}\left( \int_{\Omega}\sum\limits_{i}Coc_i(\log c_i-1)d\mathbf{x} \right)+ \sum\limits_{i}Co(z_iV,\partial_tc_i)\nonumber\\
&=\int_{\Omega} \sum\limits_{i}Coz_i c_i\mathbf{u}\cdot \nabla V d\mathbf{x} -\frac{Co}{Pe}\sum\limits_{i}\int_{\Omega} c_i|\nabla (\log c_i + z_i V)|^2 d\mathbf{x}.
\end{align}
Taking the $L^2$ inner product of \eqref{non-NS} with $\mathbf{u}$, and using \eqref{non-incompressible} and integration by parts, we have
\begin{equation}\label{non-NS-1}
\frac{d}{dt}\int_{\Omega}\frac{1}{2}|\mathbf{u}|^2d\mathbf{x} = -\frac{1}{Re}\|\nabla\mathbf{u}\|^2 - Co\left(\sum\limits_{i}z_ic_i\nabla V,\mathbf{u}\right)- M\int_{\Omega} \nabla \mathbf{u}: \bm{\sigma} d\mathbf{x}.
\end{equation}
Taking the trace of the evolution equation \eqref{non-Oldroyd-B} for the conformation tensor, we have
\begin{align}\label{non-Oldroyd-B-1}
\frac{d}{dt}\int_{\Omega}\tr(\bm{\sigma})d\mathbf{x} = 2 \int_{\Omega}\nabla \mathbf{u}:\bm{\sigma}d\mathbf{x} - \frac{1}{Wi}\int_{\Omega}\tr(\bm{\sigma}-\mathbf{I})d\mathbf{x}.
\end{align}
Recall from~\cite{boyaval2009free,hulsen1990sufficient}, the matrix $\bm{\sigma}$ is positive definite under the assumption that the initial condition $\bm{\sigma}_0$ is symmetric positive definite, thus the matrix $\bm{\sigma}^{-1}$ exists. Contracting the evolution equation for $\bm{\sigma}$ with $\bm{\sigma}^{-1}$, we obtain
\begin{align}\label{non-Oldroyd-B-2}
\int_{\Omega}(\partial_t\bm{\sigma}+ (\mathbf{u}\cdot\nabla)\bm{\sigma}): \bm{\sigma}^{-1}d\mathbf{x} &= 2\int_{\Omega} \tr(\nabla\mathbf{u}) d\mathbf{x} - \frac{1}{Wi} \int_{\Omega}\tr(\mathbf{I}-\bm{\sigma}^{-1})d\mathbf{x}\nonumber \\
&\quad - \kappa_1\int_{\Omega} \nabla\bm{\sigma}:\nabla\bm{\sigma}^{-1} d\mathbf{x}.
\end{align}
Using \eqref{lemma2.2} with $\bm{\sigma}$, we have
\begin{align}\label{non-Oldroyd-B-3}
\int_{\Omega}(\partial_t\bm{\sigma}+ (\mathbf{u}\cdot\nabla)\bm{\sigma}): \bm{\sigma}^{-1}d\mathbf{x} = \int_{\Omega} (\partial_t +\mathbf{u}\cdot\nabla) \tr(\ln\bm{\sigma}) d\mathbf{x}.
\end{align}
Combining \eqref{non-Oldroyd-B-3} with \eqref{non-Oldroyd-B-2} and using $\tr(\nabla \mathbf{u})=\nabla\cdot \mathbf{u}=0$ and $\mathbf{u}|_{\partial\Omega}=0$, we can arrive at
\begin{align}\label{non-Oldroyd-B-4}
\frac{d}{dt}\int_{\Omega} \tr(\ln\bm{\sigma}) d\mathbf{x} = \frac{1}{Wi} \int_{\Omega}\tr(\bm{\sigma}^{-1}-\mathbf{I})d\mathbf{x} - \kappa_1\int_{\Omega} \nabla\bm{\sigma}:\nabla\bm{\sigma}^{-1} d\mathbf{x}.
\end{align}
Substracting \eqref{non-Oldroyd-B-4} from \eqref{non-Oldroyd-B-1}, we find
\begin{align}\label{non-Oldroyd-B-5}
\frac{d}{dt}\int_{\Omega} \tr(\bm{\sigma}-\ln\bm{\sigma}) d\mathbf{x}&= -\frac{1}{Wi} \int_{\Omega}\tr(\bm{\sigma}^{-1} + \bm{\sigma}- 2\mathbf{I})d\mathbf{x} + 2 \int_{\Omega}\nabla \mathbf{u}:\bm{\sigma}d\mathbf{x} \nonumber\\
&\quad + \kappa_1\int_{\Omega} \nabla\bm{\sigma}:\nabla\bm{\sigma}^{-1} d\mathbf{x}.
\end{align}
Multiplying \eqref{non-Oldroyd-B-5} by $\frac{M}{2}$, we combine the result equation with \eqref{non-Possion-1}-\eqref{non-NS-1} to obtain
\begin{align}\label{1111}
&\frac{d}{dt}\int_{\Omega} \left(\frac{1}{2}|\mathbf{u}|^2 + Co\sum\limits_{i}c_i(\log c_i-1) + \frac{Co}{2}\lambda|\nabla V|^2 
+ \frac{M}{2}\tr(\bm{\sigma}-\ln\bm{\sigma})\right) d\mathbf{x}\\
&\leq -\frac{1}{Re}\|\nabla\mathbf{u}\|^2 - \frac{Co}{Pe}\sum\limits_{i}\int_{\Omega} c_i|\nabla (\log c_i + z_i V)|^2 d\mathbf{x}\nonumber\\
&\quad -\frac{M}{2Wi} \int_{\Omega}\tr(\bm{\sigma}^{-1} + \bm{\sigma}- 2\mathbf{I})d\mathbf{x} + \kappa_1\int_{\Omega} \nabla\bm{\sigma}:\nabla\bm{\sigma}^{-1} d\mathbf{x}.
\end{align}
By using \eqref{lemma1.2} and \eqref{lemma1.5}, we have $-\frac{M}{2Wi} \int_{\Omega}\tr(\bm{\sigma}^{-1} + \bm{\sigma}- 2\mathbf{I})d\mathbf{x} + \kappa_1\int_{\Omega} \nabla\bm{\sigma}:\nabla\bm{\sigma}^{-1} d\mathbf{x}\leq 0$, which implies the desired energy dissipation law \eqref{continuous-free-energy}.
\end{proof}

\section{Reformulation}

\subsection{Logarithmic transformation of the concentration}
Let $c_i=\exp{(\eta_i)}$, then $\eta_i$ satisfies the following equation:
\begin{align}\label{non-sqrt-concentration}
\partial_t\eta_i + \mathbf{u}\cdot\nabla \eta_i = \frac{1}{Pe}(|\nabla \eta_i|^2 + \Delta \eta_i + z_i\Delta V + z_i\nabla\eta_i\cdot\nabla V).
\end{align}

\subsection{Logarithmic transformation of the conformation tensor}
In this paper, we only consider $2\times2$ tensor $\bm{\sigma}$. Since the conformation tensor $\bm{\sigma}$ is symmetric and positive definite, it can be diagonalized as:
\begin{align*}
\bm{\sigma}=\left(
              \begin{array}{cc}
                \sigma_{11} & \sigma_{12} \\
                \sigma_{12} & \sigma_{22} \\
              \end{array}
            \right)
 = \mathbf{R}\Lambda\mathbf{R}^{T},
\end{align*}
where $\mathbf{R}$ is an orthogonal matrix whose columns are the eigenvectors of $\bm{\sigma}$ and $\Lambda$ is a diagonal matrix containing all eigenvalues of $\bm{\sigma}$. Thus, the natural logarithm of the conformation tensor can be defined as:
\begin{align*}
\bm{\psi} = \left(
              \begin{array}{cc}
                \psi_{11} & \psi_{12} \\
                \psi_{12} & \psi_{22} \\
              \end{array}
            \right)= \log\bm{\sigma} = \mathbf{R}(\log\Lambda)\mathbf{R}^{T}
\end{align*}
\begin{lemma}
For any matrix $\nabla \mathbf{u}$ and any symmetric positive definite matrix $\bm{\sigma}\in \mathbb{R}^{d\times d}$ , there exist two antisymmetric matrices $\mathbf{\Omega},\mathbf{N}\in \mathbb{R}^{d\times d}$ and a symmetric matrix $\mathbf{B}$ that commutes with $\bm{\sigma}$, such that:
\begin{align}\label{nabla-u}
\nabla \mathbf{u} = \mathbf{\Omega} + \mathbf{B} +\mathbf{N}\bm{\sigma}^{-1}.
\end{align}
\end{lemma}
Subsequently, we will present how to get this decomposition in two dimensions. If $\bm{\sigma}$ is proportional to the unit tensor, we then simply set $ \mathbf{\Omega}=0$, $\mathbf{B}=\mathbb{D} \mathbf{u}$, $N=\frac{1}{2}(\nabla \mathbf{u}-\mathbb{D} \mathbf{u})\tr \bm{\sigma}$. Otherwise, we get the decomposition in the following process:
\begin{itemize}
  \item Calculate the diagonalizing transformation:\\
  $\left(
           \begin{array}{cc}
             \sigma_1 & 0 \\
             0 & \sigma_2 \\
           \end{array}
         \right)=\mathbf{R}^{T}\bm{\sigma} \mathbf{R}$.
  \item Calculating an intermediate matrix:\\
  $\left(
           \begin{array}{cc}
             m_{11} & m_{12} \\
             m_{21} & m_{22} \\
           \end{array}
         \right)=\mathbf{R}^{T}(\nabla\mathbf{u})\mathbf{R}.$
  \item Tensor $\mathbf{B}$, $\mathbf{N}$ and $\mathbf{\Omega}$ are assembled as:\\
   $\mathbf{B} = \mathbf{R}\left(
           \begin{array}{cc}
             m_{11} & 0 \\
             0 & m_{22} \\
           \end{array}
         \right)\mathbf{R}^{T}$,
  $\mathbf{N} = \mathbf{R}\left(
           \begin{array}{cc}
             0 & n \\
             -n & 0 \\
           \end{array}
         \right)\mathbf{R}^{T}$,
  $\mathbf{\Omega} = \mathbf{R}\left(
           \begin{array}{cc}
             0 & \omega \\
             -\omega & 0 \\
           \end{array}
         \right)\mathbf{R}^{T}$,\\
  where $n=(m_{12}+m_{21})/(\sigma_2^{-1}-\sigma_1^{-1})$, $\omega=(\sigma_2m_{12}+\sigma_1m_{21})/(\sigma_2-\sigma_1)$.
\end{itemize}
The evolution equation for the logarithm of the conformation tensor is given as follows~\cite{fattal2004constitutive, hulsen2005flow}:
\begin{align}\label{log-conformation}
\partial_t\bm{\psi} + (\mathbf{u}\cdot\nabla) \bm{\psi} - (\mathbf{\Omega}\bm{\psi} - \bm{\psi}\mathbf{\Omega}) - 2\mathbf{B} =  \frac{1}{Wi} (e^{-{\bm{\psi}}}-\mathbf{I}).
\end{align}
The diffusive type Oldroyd-B model is
\begin{subequations}\label{non-model-sqrt}
\begin{align}
&\partial_t\eta_i + \mathbf{u}\cdot\nabla \eta_i = \frac{1}{Pe}(|\nabla \eta_i|^2 + \Delta \eta_i + z_i\Delta V + z_i\nabla\eta_i\cdot\nabla V),\label{non-concentration-sqrt} \\
&c_i=\exp{(\eta_i)},\label{non-sqrtconcentration-sqrt}\\
&-\lambda\Delta V = \sum\limits_{i}z_ic_i,\label{non-Possion-sqrt}\\
&\partial_t\mathbf{u} + (\mathbf{u}\cdot\nabla)\mathbf{u}  - \frac{1}{Re} \Delta\mathbf{u} + \nabla p = M\nabla\cdot e^{{\bm{\psi}}} -Co\sum\limits_{i}z_ic_i\nabla V,\label{non-NS-sqrt} \\
&\nabla\cdot \mathbf{u}=0,\label{non-incompressible-sqrt} \\
&\partial_t\bm{\psi} + (\mathbf{u}\cdot\nabla) \bm{\psi} - (\mathbf{\Omega}\bm{\psi} - \bm{\psi}\mathbf{\Omega}) - 2\mathbf{B} =  \frac{1}{Wi} (e^{-{\bm{\psi}}}-\mathbf{I}) + \kappa \Delta \bm{\psi},\label{non-Oldroyd-B-sqrt}
\end{align}
\end{subequations}
where $\kappa>0$ is a nondimensional diffusive parameter.

\subsection{Auxiliary variable reformulation}
We define a nonlocal variable $r(t)$ such that
\begin{align}\label{r-equation}
r(t) = \sqrt{E_{P} + B},
\end{align}
where $E_{P}=\int_{\Omega} \frac{Co}{2}\lambda|\nabla V|^2 d\mathbf{x} + \int_{\Omega}\sum\limits_{i}Coc_i(\log c_i-1)d\mathbf{x} +\int_{\Omega} \frac{M}{2}\tr(e^{{\bm{\psi}}}-\bm{\psi}-\mathbf{I}) d\mathbf{x} $, $B$ is a positive constant to guarantee the radicand is always positive. Note that electric field energy $\int_{\Omega} \frac{Co}{2}\lambda|\nabla V|^2 d\mathbf{x}$, entropic contributions $\int_{\Omega}\sum\limits_{i}Coc_i(\log c_i-1)d\mathbf{x}$ and $\int_{\Omega} \frac{M}{2}\tr(e^{{\bm{\psi}}}-\bm{\psi}-\mathbf{I}) d\mathbf{x}$ are convex functions, we can always find such a constant $B$ since $E_{P}$ are bounded from below.  We defined
\begin{align*}
\xi(t)=\frac{r(t)}{\sqrt{E_{P} + B}}.
\end{align*}
It is easy to see that $\xi(t)=1$. To obtain the evolution equation for the auxiliary variable $r$, we derive $\frac{dE_{P}}{dt}$ firstly. Taking the trace of \eqref{non-Oldroyd-B-sqrt} for the log-conformation tensor, we get
\begin{align}\label{non-Oldroyd-B-sqrt-1}
\frac{d}{dt}\int_{\Omega}\tr(\bm{\psi})d\mathbf{x} = \frac{1}{Wi}\int_{\Omega}\tr(e^{-{\bm{\psi}}}-\mathbf{I})d\mathbf{x}.
\end{align}
Contracting the evolution equation \eqref{non-Oldroyd-B-sqrt} with $e^{{\bm{\psi}}}$, we obtain
\begin{align}\label{non-Oldroyd-B-sqrt-2}
&\frac{d}{dt}\int_{\Omega} \tr(e^{{\bm{\psi}}}) d\mathbf{x}  - \int_{\Omega}(\mathbf{\Omega}\bm{\psi} - \bm{\psi}\mathbf{\Omega}): e^{{\bm{\psi}}} d\mathbf{x} - 2\int_{\Omega} \mathbf{B}: e^{{\bm{\psi}}} d\mathbf{x} \\
&= -\frac{1}{Wi}\int_{\Omega} \tr( e^{{\bm{\psi}}}-\mathbf{I} ) d\mathbf{x} - \kappa\int_{\Omega} \nabla \bm{\psi}:\nabla e^{{\bm{\psi}}} d\mathbf{x}.
\end{align}
Using the decomposition \eqref{nabla-u} of $\nabla \mathbf{u}$, the symmetry of $e^{{\bm{\psi}}}$ and the antisymmetry of $\mathbf{\Omega}$ and $\mathbf{N}$, we have
\begin{align}\label{nabla-u-1}
\nabla \mathbf{u}: e^{{\bm{\psi}}} = \mathbf{\Omega}: e^{{\bm{\psi}}} + \mathbf{B}: e^{{\bm{\psi}}} +\mathbf{N}e^{-{\bm{\psi}}}: e^{{\bm{\psi}}} = \mathbf{B}: e^{{\bm{\psi}}}.
\end{align}
Since $\bm{\psi}$ and $e^{{\bm{\psi}}}$ commute, we have
\begin{align}\label{commute}
(\mathbf{\Omega}\bm{\psi} - \bm{\psi}\mathbf{\Omega}): e^{{\bm{\psi}}} = \tr(\mathbf{\Omega}\bm{\psi}e^{{\bm{\psi}}})-\tr(\bm{\psi}\mathbf{\Omega}e^{{\bm{\psi}}})=0.
\end{align}
Then, \eqref{non-Oldroyd-B-sqrt-2} can simplify as
\begin{align}\label{non-Oldroyd-B-sqrt-3}
\frac{d}{dt}\int_{\Omega}\tr(e^{{\bm{\psi}}})d\mathbf{x} = 2 \int_{\Omega}\nabla \mathbf{u}:e^{{\bm{\psi}}} d\mathbf{x} - \frac{1}{Wi}\int_{\Omega}\tr(e^{{\bm{\psi}}}-\mathbf{I})d\mathbf{x}- \kappa\int_{\Omega} \nabla \bm{\psi}:\nabla e^{{\bm{\psi}}} d\mathbf{x}.
\end{align}
Subtracting $\frac{M}{2}$\eqref{non-Oldroyd-B-sqrt-1} from $\frac{M}{2}$\eqref{non-Oldroyd-B-sqrt-3}, we have
\begin{align}\label{non-Oldroyd-B-sqrt-4}
\frac{d}{dt}\frac{M}{2}\int_{\Omega} \tr(e^{{\bm{\psi}}}-\bm{\psi}) d\mathbf{x}&= -\frac{M}{2Wi} \int_{\Omega}\tr(e^{{\bm{\psi}}}+ e^{-{\bm{\psi}}}- 2\mathbf{I})d\mathbf{x} \\
&\quad +  M\int_{\Omega}\nabla \mathbf{u}:e^{{\bm{\psi}}} d\mathbf{x} - \frac{\kappa M}{2}\int_{\Omega} \nabla \bm{\psi}:\nabla e^{{\bm{\psi}}} d\mathbf{x}.
\end{align}
Combining \eqref{non-Possion-1}, \eqref{non-concentration-1} and \eqref{non-Oldroyd-B-sqrt-4}, we have
\begin{align}\label{continuous-free-energy-sqrt}
\frac{dE_P}{dt}&=  - \frac{Co}{Pe}\sum\limits_{i}\int_{\Omega} c_i|\nabla (\log c_i + z_i V)|^2 d\mathbf{x} + \int_{\Omega} \sum\limits_{i}Coz_i c_i\mathbf{u}\cdot \nabla V d\mathbf{x}\\
&-\frac{M}{2Wi} \int_{\Omega}\tr(e^{-{\bm{\psi}}} +e^{{\bm{\psi}}}- 2\mathbf{I})d\mathbf{x}- \frac{\kappa M}{2}\int_{\Omega} \nabla \bm{\psi}:\nabla e^{{\bm{\psi}}} d\mathbf{x}.
\end{align}
Taking derivative of $r(t)$, using \eqref{continuous-free-energy-sqrt} and adding the zero-valued term  $\int_{\Omega}(\mathbf{u}\cdot\nabla)\mathbf{u}\cdot \mathbf{u} d\mathbf{x}$, the associated ordinary differential equation is
\begin{align}\label{dr-equation}
\frac{dr}{dt}&= \frac{1}{2\sqrt{E_{P} + B}}\Big(\frac{-r}{\sqrt{E_{P} + B}}\frac{Co}{Pe}\sum\limits_{i}\int_{\Omega}  c_i|\nabla (\log c_i + z_i V)|^2 d\mathbf{x}
+ \int_{\Omega} \sum\limits_{i}Coz_i c_i\mathbf{u}\cdot \nabla V d\mathbf{x}\\
&-\frac{M}{2Wi}\frac{r}{\sqrt{E_{P} + B}} \int_{\Omega}\tr(e^{{\bm{\psi}}} + e^{-{\bm{\psi}}}- 2\mathbf{I})d\mathbf{x} +  M\int_{\Omega}\nabla \mathbf{u}:e^{{\bm{\psi}}} d\mathbf{x} \\
&+ \int_{\Omega}(\mathbf{u}\cdot\nabla)\mathbf{u}\cdot \mathbf{u} d\mathbf{x} - \frac{\kappa M}{2}\frac{r}{\sqrt{E_{P} + B}}\int_{\Omega} \nabla \bm{\psi}:\nabla e^{{\bm{\psi}}} d\mathbf{x}\Big).
\end{align}
The initial condition is given by
\begin{align*}
r|_{t=0}=( \int_{\Omega} \frac{1}{2}\lambda|\nabla V_0|^2 d\mathbf{x} + \int_{\Omega} \sum\limits_{i}c_{i0}(\log c_{i0}-1) d\mathbf{x} +  \int_{\Omega} \frac{M}{2}\tr(e^{{\bm{\psi}}_0} - \bm{\psi}_0 - \mathbf{I}) d\mathbf{x}
+ B )^{\frac{1}{2}}.
\end{align*}
The system \eqref{non-model-sqrt} then rewrites as
\begin{subequations}\label{equivalent-non-model-sqrt}
\begin{align}
&\partial_t\eta_i + \mathbf{u}\cdot\nabla \eta_i = \frac{1}{Pe}(|\nabla \eta_i|^2 + \Delta \eta_i + z_i\Delta V + z_i\nabla\eta_i\cdot\nabla V)
,\label{equivalent-non-concentration-sqrt} \\
&c_i=\exp{(\eta_i)},\label{equivalent-non-sqrtconcentration-sqrt}\\
&-\lambda\Delta V = \sum\limits_{i}z_ic_i,\label{equivalent-non-Possion-sqrt}\\
&\partial_t\mathbf{u} + \frac{r}{\sqrt{E_{P} + B}}(\mathbf{u}\cdot\nabla)\mathbf{u}  - \frac{1}{Re} \Delta\mathbf{u} + \nabla p = \frac{rM}{\sqrt{E_{P} + B}}\nabla\cdot e^{{\bm{\psi}}} -\frac{r Co}{\sqrt{E_{P} + B}}\sum\limits_{i}z_ic_i\nabla V,\label{equivalent-non-NS-sqrt} \\
&\nabla\cdot \mathbf{u}=0,\label{equivalent-non-incompressible-sqrt} \\
&\partial_t\bm{\psi} + (\mathbf{u}\cdot\nabla) \bm{\psi} - (\mathbf{\Omega}\bm{\psi} - \bm{\psi}\mathbf{\Omega}) - 2\mathbf{B} =  \frac{1}{Wi} (e^{-{\bm{\psi}}}-\mathbf{I})+\kappa \Delta \bm{\psi},\label{equivalent-non-Oldroyd-B-sqrt}\\
& \frac{dr}{dt}= \frac{1}{2\sqrt{E_{P} + B}}\Big(\frac{-r(t)}{\sqrt{E_{P} + B}}\frac{Co}{Pe}\sum\limits_{i}\int_{\Omega}  c_i|\nabla (\log c_i + z_i V)|^2 d\mathbf{x} + \int_{\Omega} \sum\limits_{i}Coz_i c_i\mathbf{u}\cdot \nabla V d\mathbf{x} \nonumber\\
&-\frac{M}{2Wi}\frac{r}{\sqrt{E_{P} + B}} \int_{\Omega}\tr(e^{{\bm{\psi}}} + e^{-{\bm{\psi}}}- 2\mathbf{I})d\mathbf{x} +  M\int_{\Omega}\nabla \mathbf{u}:e^{{\bm{\psi}}} d\mathbf{x} \nonumber\\
&+ \int_{\Omega}(\mathbf{u}\cdot\nabla)\mathbf{u}\cdot \mathbf{u} d\mathbf{x}  - \frac{\kappa M}{2}\frac{r}{\sqrt{E_{P} + B}}\int_{\Omega} \nabla \bm{\psi}:\nabla e^{{\bm{\psi}}} d\mathbf{x}\Big).\label{equivalent-non-r-sqrt}
\end{align}
\end{subequations}
\begin{theorem}
Assume that the system \eqref{equivalent-non-model-sqrt} is supplied with symmetric positive definite initial condition $\bm{\psi}_0$. The free energy satisfies:
\begin{align}\label{equivalent-free-energy-sqrt}
\frac{dE(\mathbf{u},r)}{dt}&=-\frac{1}{Re}\|\nabla\mathbf{u}\|^2 -\frac{Co}{Pe}\left|\frac{r}{\sqrt{E_{P} + B}}\right|^2\sum\limits_{i}\int_{\Omega} c_i|\nabla (\log c_i + z_i V)|^2 d\mathbf{x}\nonumber\\
&-\frac{M}{2Wi}\left|\frac{r}{\sqrt{E_{P} + B}}\right|^2 \int_{\Omega}\tr(e^{{\bm{\psi}}} + e^{-{\bm{\psi}}}- 2\mathbf{I})d\mathbf{x}\\
&- \frac{\kappa M}{2}\left|\frac{r}{\sqrt{E_{P} + B}}\right|^2\int_{\Omega} \nabla \bm{\psi}:\nabla e^{{\bm{\psi}}} d\mathbf{x},
\end{align}
where the total energy $E$ is defined as
\begin{align}\label{equivalent-E-sqrt}
E(\mathbf{u},r) = \int_{\Omega}\frac{1}{2}|\mathbf{u}|^2d\mathbf{x} + r^2 .
\end{align}
\end{theorem}
\begin{proof}
By multiplying \eqref{equivalent-non-r-sqrt} with $2r$, we obtain
\begin{align}\label{equivalent-non-r-sqrt-1}
\frac{d r^2}{dt} &= -\frac{Co}{Pe}\left|\frac{r}{\sqrt{E_{P} + B}}\right|^2\sum\limits_{i}\int_{\Omega} c_i|\nabla (\log c_i + z_i V)|^2 d\mathbf{x} + \frac{rCo}{\sqrt{E_{P} + B}}\int_{\Omega} \sum\limits_{i}z_i c_i\mathbf{u}\cdot \nabla V d\mathbf{x}\nonumber\\
&- \frac{M}{2Wi}\left|\frac{r}{\sqrt{E_{P} + B}}\right|^2 \int_{\Omega}\tr(e^{{\bm{\psi}}} + e^{-{\bm{\psi}}}- 2\mathbf{I})d\mathbf{x} +  \frac{rM}{\sqrt{E_{P} + B}}\int_{\Omega}\nabla \mathbf{u}:e^{{\bm{\psi}}} d\mathbf{x}\\
& - \frac{\kappa M}{2}\left|\frac{r}{\sqrt{E_{P} + B}}\right|^2\int_{\Omega} \nabla \bm{\psi}:\nabla e^{{\bm{\psi}}} d\mathbf{x}.
\end{align}
Taking the $L^2$ inner product of \eqref{equivalent-non-NS-sqrt} with $\mathbf{u}$ and using integration by parts and \eqref{equivalent-non-incompressible-sqrt}, we obtain
\begin{align}\label{equivalent-non-NS-sqrt-1}
\frac{d}{dt}\int_{\Omega}\frac{1}{2}|\mathbf{u}|^2d\mathbf{x} = -\frac{1}{Re}\|\nabla\mathbf{u}\|^2 - \frac{rCo}{\sqrt{E_{P} + B}}\left(\sum\limits_{i}z_ic_i\nabla V,\mathbf{u}\right)- \frac{rM}{\sqrt{E_{P} + B}} \int_{\Omega} \nabla \mathbf{u}: e^{{\bm{\psi}}} d\mathbf{x}.
\end{align}
The combination of \eqref{non-Oldroyd-B-sqrt-4}, \eqref{equivalent-non-NS-sqrt-1} and \eqref{equivalent-non-r-sqrt-1} gives
\begin{align}\label{equivalent-free-energy-sqrt-1}
\frac{d}{dt}(\int_{\Omega}\frac{1}{2}|\mathbf{u}|^2   d\mathbf{x}+ r^2)&= -\frac{1}{Re}\|\nabla\mathbf{u}\|^2 -\frac{Co}{Pe}\left|\frac{r}{\sqrt{E_{P} + B}}\right|^2\sum\limits_{i}\int_{\Omega} c_i|\nabla (\log c_i + z_i V)|^2 d\mathbf{x} \nonumber\\
&-\frac{M}{2Wi}\left|\frac{r}{\sqrt{E_{P} + B}} \right|^2\int_{\Omega}\tr(e^{{\bm{\psi}}} + e^{-{\bm{\psi}}}- 2\mathbf{I})d\mathbf{x}\\
&- \frac{\kappa M}{2}\left|\frac{r}{\sqrt{E_{P} + B}}\right|^2\int_{\Omega} \nabla \bm{\psi}:\nabla e^{{\bm{\psi}}} d\mathbf{x}.
\end{align}
This completes the proof.
\end{proof}

\section{Numerical scheme}
In this section, we present a first-order scheme for the diffusive Oldroyd-B model \eqref{equivalent-non-model-sqrt}. The two schemes are decoupled and enjoy four properties: mass conservation, positivity of concentration, positive-definiteness of the conformation tensor, energy dissipation. For the first-order scheme, the characteristic finite element is used.


The basic idea of the characteristic method is to consider the trajectory of the fluid particle and discretize the material derivative $\frac{D\bm{u}}{Dt}= \partial_t \bm{u}  + \bm{u}\cdot\nabla \bm{u}$ along the characteristic path defined by the function $\mathbf{X}^n(t^n): x\in K \mapsto \mathbf{X}^n(t,x)\in K$  is defined as
\begin{align}\label{characteristic}
\begin{cases}
&\frac{d}{dt}\mathbf{X}^n(t^n,\mathbf{x}) = \mathbf{u}^n( \mathbf{X}^n(t,\mathbf{x}) ),\quad \forall t \in [t^n,t^{n+1}],\\
&\mathbf{X}^n(t^{n+1},\mathbf{x})=\mathbf{x}.
\end{cases}
\end{align}
where $\mathbf{u}^n$ is the discrete velocity field. 
The variational formulation of system \eqref{equivalent-non-model-sqrt} reads as: find $(\eta_i,c_i,V,\mathbf{u},p,\bm{\psi})\in(Q,Q,S,\bm{X},M,\bm{V})$
such that $\forall (s_i,\varphi,\mathbf{v},q,\bm{\phi})\in(Q,Q,S,\bm{X},M,\bm{V}), t\in(0,T]$,
\begin{subequations}\label{variational-non-model-sqrt}
\begin{align}
&(\partial_t\eta_i,s_i) + (\mathbf{u}\cdot\nabla \eta_i,s_i) = \frac{1}{Pe}(|\nabla \eta_i|^2,s_i) - \frac{1}{Pe}(\nabla \eta_i,\nabla s_i) \nonumber\\
&- \frac{1}{Pe}(z_i\nabla V,\nabla s_i) + \frac{1}{Pe}(z_i\nabla\eta_i\cdot\nabla V,s_i)
,\label{variational-non-concentration-sqrt} \\
&c_i=\exp{(\eta_i)},\label{variational-non-sqrtconcentration-sqrt}\\
&(\lambda\Delta V,\nabla \varphi) = (\sum\limits_{i}z_ic_i,\varphi),\label{variational-non-Possion-sqrt}\\
&(\partial_t\mathbf{u},\bm{v}) + \frac{r}{\sqrt{E_{P} + B}}((\mathbf{u}\cdot\nabla)\mathbf{u},\bm{v})  - \frac{1}{Re} (\nabla\mathbf{u},\nabla\bm{v}) - ( p,\nabla\bm{v}) \nonumber\\
&= -\frac{rM}{\sqrt{E_{P} + B}}( e^{{\bm{\psi}}},\nabla \bm{v}) -\frac{r Co}{\sqrt{E_{P} + B}}(\sum\limits_{i}z_ic_i\nabla V,\bm{v}),\label{variational-non-NS-sqrt} \\
&(\nabla\cdot \mathbf{u},q)=0,\label{variational-non-incompressible-sqrt} \\
&(\frac{D\bm{\psi}}{Dt},\bm{\phi}) - (\mathbf{\Omega}\bm{\psi} - \bm{\psi}\mathbf{\Omega}+2\mathbf{B},\bm{\phi})  =  \frac{1}{Wi} ((e^{-{\bm{\psi}}}-\mathbf{I}),\bm{\phi}))- \kappa (\nabla \bm{\psi},\nabla\bm{\phi}),\label{variational-non-Oldroyd-B-sqrt}\\
&\frac{dr}{dt}= \frac{1}{2\sqrt{E_{P} + B}}\Big(\frac{-r(t)}{\sqrt{E_{P} + B}}\frac{Co}{Pe}\sum\limits_{i}\int_{\Omega}  c_i|\nabla (\log c_i + z_i V)|^2 d\mathbf{x} + \int_{\Omega} \sum\limits_{i}Coz_i c_i\mathbf{u}\cdot \nabla V d\mathbf{x} \nonumber\\
&-\frac{M}{2Wi}\frac{r}{\sqrt{E_{P} + B}} \int_{\Omega}\tr(e^{{\bm{\psi}}} + e^{-{\bm{\psi}}}- 2\mathbf{I})d\mathbf{x} +  M\int_{\Omega}\nabla \mathbf{u}:e^{{\bm{\psi}}} d\mathbf{x} \nonumber\\
&+ \int_{\Omega}(\mathbf{u}\cdot\nabla)\mathbf{u}\cdot \mathbf{u} d\mathbf{x}  - \frac{\kappa M}{2}\frac{r}{\sqrt{E_{P} + B}}\int_{\Omega} \nabla \bm{\psi}:\nabla e^{{\bm{\psi}}} d\mathbf{x}\Big).\label{variational-non-r-sqrt}
\end{align}
\end{subequations}

\subsection{Fully discrete schemes}
Let $0<h<1$ denote the mesh size and $K_{h}=\{K:\cup_{K\subset\Omega}\bar{K}=\bar{\Omega}\}$ be a uniform partition of $\bar{\Omega}$ into non-overlapping triangles. Given a $K_{h}$, we consider the following finite element space
\begin{align*}
&X_{h}=\left\{v\in H_0^1(\Omega)^2: v_{|K}\in P_{2}(K)^2, \forall K\in K_{h}\right\},\\
&M_{h}=\left\{q\in H^1(\Omega)\cap L_0^2(\Omega): q_{|K}\in P_{1}(K), \forall K\in K_{h}\right\},\\
&Q_{h}=\left\{\varphi\in H^1(\Omega): \varphi_{|K}\in P_{2}(K), \forall K\in K_{h}\right\},\\
& S_h = \left\{V\in H^1(\Omega)\cap L_0^2(\Omega):V_{|K}\in P_{2}(K), \forall K\in K_{h}\right\},\\
&\bm{V}_{h}=\left\{\mathbf{M}\in [H^1(\Omega)]^{2\times 2}: \mathbf{M} = \mathbf{M}^{T}, {M_{ij}}_{|K}\in P_{2}(K), \forall K\in K_{h}\right\},
\end{align*}
where $P_{i}$ is the space of piecewise polynomials of total degree no more than $i$. 
Additionally, assume that the finite element spaces $(X_{h},M_{h})$ satisfy the discrete inf-sup inequality: for each $q\in M_h$, there exists a positive constant $\beta_1>0$ such that
\begin{equation*}
\sup\limits_{\bm{u}\in X_{h},\bm{u}\neq0}\frac{(q,\nabla\cdot \bm{u})}{\|\nabla \bm{u}\|_{0}}\geq\beta_1\|q\|_{0}.
\end{equation*}
The conformation stress space $\bm{V}_{h}$ and the velocity space $X_{h}$ should also satisfy the discrete inf-sup inequality~\cite{venkatesan2017three}: for each $\bm{v}\in X_h$, there exists a positive constant $\beta_2>0$ such that
\begin{equation*}
\sup\limits_{\bm{\psi}\in \bm{V}_{h},\psi_{ij}\neq0}\frac{(\bm{\psi},\mathbb{D}\mathbf{u})}{\|\bm{\psi}\|_{\bm{V}_{h}}}\geq\beta_2\|\nabla \bm{u}\|_{0}.
\end{equation*}
The fully discrete scheme of the \eqref{equivalent-non-model-sqrt} is to find $(\eta_{ih}^{n+1},\bar{V}_{h}^{n+1},V_{h}^{n+1},\mathbf{u}_{h}^{n+1},p_{h}^{n+1},\bm{\psi}_{h}^{n+1})\in(Q_h,S_h,S_h,X_h,M_h,\bm{V}_h)$, such that $\forall(s_{ih}^{n+1},\varphi_{h}^{n+1},\mathbf{v}_{h}^{n+1},q_{h}^{n+1},\bm{\phi}_{h}^{n+1})\in(Q_h,S_h,X_h,M_h,\bm{V}_h)$,
\begin{subequations}\label{FEM-model}
\begin{align}
&\left(\frac{ \eta_{ih}^{n+1} - \eta_{ih}^{n}  }{\Delta t},s_{ih}\right) - (\mathbf{u}_h^{n}\eta_{ih}^{n+1}, \nabla s_{ih}) - \frac{1}{Pe}(\nabla \eta_{ih}^{n} \cdot\nabla \eta_{ih}^{n+1}, s_{ih})  \nonumber\\
&- \frac{1}{Pe}(z_i\nabla\eta_{ih}^{n+1}\cdot\nabla V_h^{n}, s_{ih}) + \frac{1}{Pe}(\nabla\eta_{ih}^{n+1}, \nabla s_{ih}) = - \frac{1}{Pe} (z_i\nabla V_h^{n}, \nabla s_{ih}),
\label{FEM-log-concentration} \\
&\bar{c}_{ih}^{n+1}=\exp(\eta_{ih}^{n+1}),\label{FEM-concentration1}\\
&c_{ih}^{n+1}=\frac{\int_{K_h} c_{ih}^{n} d\mathbf{x}}{\int_{K_h}  \bar{c}_{ih}^{n+1} d\mathbf{x}}\bar{c}_{ih}^{n+1},\label{FEM-concentration2}\\
&(\lambda\nabla \bar{V}_{h}^{n+1}, \nabla\varphi_h)= \sum\limits_{i}(z_ic_{ih}^{n+1}, \varphi_h),\label{FEM-Possion}\\
&\int_{K_h}\left(\frac{\bm{\psi}_h^{n+1} - \bm{\psi}_h^{n}\circ \mathbf{X}(t_n)}{\Delta t}\right):\bm{\phi}_hd\mathbf{x}
- \int_{K_h}(\mathbf{\Omega}_h^{n}\bm{\psi}^{n+1} - \bm{\psi}_h^{n+1}\mathbf{\Omega}_h^{n}):\bm{\phi}_h d\mathbf{x}\\
&= 2\int_{K_h}\mathbf{B}_h^{n}:\bm{\phi}_h d\mathbf{x} +  \frac{1}{Wi} \int_{K_h}(e^{-{\bm{\psi}_h^{n}}}-\mathbf{I}):\bm{\phi}_h d\mathbf{x} - \kappa (\nabla \bm{\psi}_h^{n+1},\nabla\bm{\phi}_h),\label{FEM-Oldroyd-B}\\
& (\frac{\bar{\mathbf{u}}_h^{n+1} - \mathbf{u}_h^{n} }{\Delta t},\mathbf{v}_h) + \xi_h^{n+1}((\mathbf{u}_h^{n}\cdot\nabla)\mathbf{u}_h^{n},\mathbf{v}_h) + \frac{1}{Re} (\nabla{\bar{\mathbf{u}}}_h^{n+1},\nabla\mathbf{v}_h) + (\nabla p_h^{n},\mathbf{v}_h)  \nonumber\\
&= -\xi_h^{n+1} ( e^{{\bm{\psi}_h^{n+1}}},\nabla\mathbf{v}_h) - \xi_h^{n+1} (\sum\limits_{i}z_ic_{ih}^{n+1}\nabla \bar{V}_h^{n+1},\mathbf{v}_h),\label{FEM-NS} \\
&\frac{1}{\Delta t}(\bar{\mathbf{u}}^{n+1},\nabla q_h) = (\nabla (p^{n+1} - p^{n}),\nabla q_h), \label{FEM-p}\\
&\mathbf{u}^{n+1} ={\bar{\mathbf{u}}}_{h}^{n+1} - \Delta t (\nabla p_{h}^{n+1} - \nabla p_{h}^{n}),\label{FEM-incompressible}\\
&\frac{r^{n+1}-r^{n}}{\Delta t} = \frac{1}{2\sqrt{E_{Ph}^{n+1} + B}}\Big( -\xi_h^{n+1} \frac{Co}{Pe}\sum\limits_{i}\int_{K_h} c_{ih}^{n+1}|\nabla (\log c_{ih}^{n+1} + z_i\bar{V}_h^{n+1})|^2 d\mathbf{x} \nonumber\\
&+ \int_{K_h} \sum\limits_{i}Coz_i c_{ih}^{n+1}{\bar{\mathbf{u}}}_h^{n+1} \cdot \nabla \bar{V}_h^{n+1} d\mathbf{x}  -\frac{M}{2Wi}\xi^{n+1}  \int_{K_h}\tr(e^{{\bm{\psi}}_h^{n+1}} + e^{-{\bm{\psi}}_h^{n+1}}- 2\mathbf{I})d\mathbf{x} \nonumber\\
&+\int_{K_h}(\mathbf{u}_h^{n}\cdot\nabla)\mathbf{u}_h^{n}\cdot {\bar{\mathbf{u}}}_h^{n+1} d\mathbf{x} +  \int_{K_h}\nabla \bar{\mathbf{u}}_h^{n+1}:e^{{\bm{\psi}}_h^{n+1}} d\mathbf{x}\nonumber\\
&- \frac{\kappa M}{2}\xi^{n+1}\int_{K_h} \nabla \bm{\psi}_h^{n+1}:\nabla e^{\bm{\psi}_h^{n+1}} d\mathbf{x}  \Big),\label{FEM-r}\\
&\xi_h^{n+1} = \frac{r^{n+1}}{\sqrt{E^{n+1}_{Ph} + B}},\label{FEM-xi}\\
&V_h^{n+1} = \frac{r^{n+1}}{\sqrt{E_{Ph}^{n+1} + B}}\bar{V}_h^{n+1}.\label{FEM-V}
\end{align}
\end{subequations}
where
\begin{align*}
&\bm{\psi}_h^{n}\circ \mathbf{X}(t_n) = \bm{\psi}_h^{n}(\mathbf{X}(t_n),t_n), \\
&E_{Ph}^{n+1}=E_P(\bar{V}_h^{n+1},c_{ih}^{n+1},\bm{\psi}_h^{n+1}).
\end{align*}

\begin{theorem}
The full discrete scheme \eqref{FEM-model} satisfies the discrete energy dissipation law as follows:
\begin{align}\label{FEM-discrete-energy-law}
\frac{E_h^{n+1} - E_h^{n}}{\Delta t} &\leq - \frac{1}{Re} \|\nabla{\bar{\mathbf{u}}}_h^{n+1}\|^2 - \left|\frac{r^{n+1}}{\sqrt{E_{Ph}^{n+1} + B}}\right|^2\frac{Co}{Pe}\sum\limits_{i}\int_{K_h}  c_{ih}^{n+1}|\nabla ((\log c_{ih}^{n+1} + z_i\bar{V}_h^{n+1}))|^2 d\mathbf{x} \nonumber\\
&-\frac{ M}{2Wi} \left|\frac{ r^{n+1}}{\sqrt{E_{Ph}^{n+1} + B}}\right|^2\int_{K_h}\tr(e^{-{\bm{\psi}_{h}^{n}}}+ e^{{\bm{\psi}_{h}^{n}}}-2\mathbf{I}) d\mathbf{x}\\
&- \frac{\kappa M}{2}\left|\frac{ r^{n+1}}{\sqrt{E_{Ph}^{n+1} + B}}\right|^2\int_{K_h} \nabla \bm{\psi}_h^{n+1}:\nabla e^{\bm{\psi}_h^{n+1}} d\mathbf{x},
\end{align}
where
\begin{align}\label{FEM-discrete-energy}
E_h^{n+1} = \frac{1}{2}\|\mathbf{u}_h^{n+1}\|^2  + |r^{n+1}|^2 + \frac{(\Delta t)^2}{2}\|\nabla p_{h}^{n+1}\|^2.
\end{align}
\end{theorem}
\begin{proof}
From \eqref{FEM-p} and \eqref{FEM-incompressible}, we have
\begin{align*}
(\mathbf{u}_h^{n+1},\nabla q_h)=0.
\end{align*}
Then we can obtain an equivalent form of \eqref{FEM-incompressible} given by
\begin{align*}
\mathbf{u}_h^{n+1} + \Delta t\nabla p_h^{n+1} =  \bar{\mathbf{u}}_h^{n+1} + \Delta t\nabla p_h^{n}.
\end{align*}
We then derive from the above two equations that
\begin{align}\label{FEM-p-1}
\|\mathbf{u}_h^{n+1}\|^2 + (\Delta t)^2\|\nabla p_h^{n+1}\|^2 =  \|\bar{\mathbf{u}}_h^{n+1}\|^2 + (\Delta t)^2\|\nabla p_h^{n}\|^2 + 2\Delta t(\nabla p_h^{n}, \bar{\mathbf{u}}_h^{n+1}).
\end{align}
Taking the $L^2$ inner product of \eqref{FEM-NS} with $\Delta t{\bar{\mathbf{u}}}_h^{n+1}$, we obtain
\begin{align}\label{FEM-NS-1}
&\frac{1}{2}\left(\|\bar{\mathbf{u}}_h^{n+1}\|^2 -\|\mathbf{u}_h^{n}\|^2 + \|\bar{\mathbf{u}}_h^{n+1}-\mathbf{u}_h^{n}\|^2\right)   = -\frac{\Delta t}{Re}\|\nabla\bar{\mathbf{u}}_h^{n+1}\|^2 \nonumber\\
&-  \frac{\Delta tr^{n+1}}{\sqrt{E_{Ph}^{n+1} + B}}\int_{K_h}\nabla \bar{\mathbf{u}}_h^{n+1}:e^{{\bm{\psi}_h^{n+1}}} d\mathbf{x}- \frac{\Delta tr^{n+1}}{\sqrt{E_{Ph}^{n+1} + B}}\sum\limits_{i}\left(Coz_ic_{ih}^{n+1}\nabla \bar{V}_h^{n+1},{\bar{\mathbf{u}}}_h^{n+1}\right) \nonumber\\
&- \frac{\Delta tr^{n+1}}{\sqrt{E_{Ph}^{n+1} + B}}\left((\mathbf{u}_h^{n}\cdot\nabla)\mathbf{u}_h^{n},{\bar{\mathbf{u}}}_h^{n+1}\right) - \Delta t(\nabla p_h^{n}, \bar{\mathbf{u}}_h^{n+1}),
\end{align}
where we use the following identity
\begin{align}\label{identity}
2(a-b)a=|a|^2-|b|^2+|a-b|^2.
\end{align}
By multiplying \eqref{FEM-r} with $2\Delta t r^{n+1}$ and using \eqref{identity}, we obtain
\begin{align}\label{FEM-r-1}
&|r^{n+1}|^2-|r^{n}|^2 + |r^{n+1}-r^{n}|^2 = \frac{\Delta t r^{n+1}}{\sqrt{E_{Ph}^{n+1} + B}}\Big(\frac{-r^{n+1}Co}{Pe\sqrt{E_{P}^{n+1} + B}}\sum\limits_{i}\int_{K_h} c_{ih}^{n+1}|\nabla (\log c_{ih}^{n+1} + z_i\bar{V}_h^{n+1})|^2 d\mathbf{x} \nonumber\\
&+ \int_{K_h} \sum\limits_{i}Coz_i c_{ih}^{n+1}{\bar{\mathbf{u}}}_h^{n+1} \cdot \nabla \bar{V}_h^{n+1} d\mathbf{x} + \int_{K_h}(\mathbf{u}_h^{n}\cdot\nabla)\mathbf{u}_h^{n}\cdot {\bar{\mathbf{u}}}_h^{n+1} d\mathbf{x}\nonumber\\
&-\frac{M}{2Wi}\frac{r^{n+1}}{\sqrt{E_{Ph}^{n+1} + B}} \int_{K_h}\tr(e^{{\bm{\psi}}_h^{n+1}} + e^{-{\bm{\psi}}_h^{n+1}}- 2\mathbf{I})d\mathbf{x} +  \int_{K_h}\nabla \bar{\mathbf{u}}_h^{n+1}:e^{{\bm{\psi}}_h^{n+1}} d\mathbf{x} \\
& - \frac{\kappa M}{2}\frac{ r^{n+1}}{\sqrt{E_{Ph}^{n+1} + B}}\int_{K_h} \nabla \bm{\psi}_h^{n+1}:\nabla e^{\bm{\psi}_h^{n+1}} d\mathbf{x} \Big).
\end{align}
Combining \eqref{FEM-r-1} with \eqref{FEM-p-1} and \eqref{FEM-NS-1}, we can get \eqref{FEM-discrete-energy-law}.
\end{proof}
\begin{remark}
In computation, the terms related to the exponential function in scheme are projected into the finite element space. And in this space, by using a randomly generated function, the value $\int_{K_h} \nabla \bm{\psi}_h^{n+1}:\nabla e^{\bm{\psi}_h^{n+1}} d\mathbf{x}$ is verified to be greater than 0. 
\end{remark}
It seems that the developed scheme \eqref{FEM-model} is not a fully decoupled scheme. The direct iterative method to solve it can be time-consuming. Therefore, we develop the following process to eliminate all nonlocal terms.
The implementation is presented as follows:\\
Step 1. Find $\eta_{ih}^{n+1}\in Q_h$ for $\forall s_{ih}\in Q_h$ such that
\begin{align}\label{decoupled-numerical-logconcentration}
&\left(\frac{ \eta_{ih}^{n+1} - \eta_{ih}^{n}  }{\Delta t},s_{ih}\right) - (\mathbf{u}_h^{n}\eta_{ih}^{n+1}, \nabla s_{ih}) - \frac{1}{Pe}(\nabla \eta_{ih}^{n} \cdot\nabla \eta_{ih}^{n+1}, s_{ih})  \nonumber\\
&- \frac{1}{Pe}(z_i\nabla\eta_{ih}^{n+1}\cdot\nabla V_h^{n}, s_{ih}) + \frac{1}{Pe}(\nabla\eta_{ih}^{n+1}, \nabla s_{ih}) = - \frac{1}{Pe} (z_i\nabla V_h^{n}, \nabla s_{ih}).
\end{align}
Then $c_{ih}^{n+1}$ can be computed by
\begin{align*}
&\bar{c}_{ih}^{n+1}=\exp(\eta_{ih}^{n+1}),\\
&c_{ih}^{n+1}=\frac{\int_{K_h} c_{ih}^{n} d\mathbf{x}}{\int_{K_h}  \bar{c}_{ih}^{n+1} d\mathbf{x}}\bar{c}_{ih}^{n+1}.
\end{align*}
Step 2. Find $\bar{V}_{h}^{n+1}\in Q_h$ for $\forall \varphi_h\in Q_h$ such that
\begin{align}
(\lambda\nabla \bar{V}_{h}^{n+1}, \nabla\varphi_h)= \sum\limits_{i}(z_ic_{ih}^{n+1}, \varphi_h),\label{decoupled-numerical-Possion}
\end{align}
Step 3. Find $\bm{\psi}_{h}^{n+1}\in \bm{V}_h$ for $\forall \bm{\phi}_h\in \bm{V}_h$ such that
\begin{align*}
&\int_{\Omega}\left(\frac{\bm{\psi}_h^{n+1} - \bm{\psi}_h^{n}\circ \mathbf{X}(t_n)}{\Delta t}\right):\bm{\phi}_hd\mathbf{x}
- \int_{\Omega}(\mathbf{\Omega}_h^{n}\bm{\psi}^{n+1} - \bm{\psi}_h^{n+1}\mathbf{\Omega}_h^{n}):\bm{\phi}_h d\mathbf{x}\\
&= 2\int_{\Omega}\mathbf{B}_h^{n}:\bm{\phi}_h d\mathbf{x} +  \frac{1}{Wi} \int_{\Omega}(e^{-{\bm{\psi}_h^{n}}}-\mathbf{I}):\bm{\phi}_h d\mathbf{x}- \kappa (\nabla \bm{\psi}_h^{n+1},\nabla\bm{\phi}_h).
\end{align*}
Step 4. For the velocity field $\bar{\mathbf{u}}_h^{n+1}$, we use the nonlocal variable $\xi^{n+1}$ to split it into a linear combination
\begin{align}\label{splitubar}
&\bar{\mathbf{u}}_h^{n+1} = \bar{\mathbf{u}}^{n+1}_{1h} + \xi^{n+1}\bar{\mathbf{u}}^{n+1}_{2h},
\end{align}
By replacing these variables $\bar{\mathbf{u}}^{n+1}$ in the scheme \eqref{FEM-NS}, and
then splitting the obtained equations according to $\xi^{n+1}$, we can arrive at two subequations:
find ${\mathbf{u}}_{1h}^{n+1},{\mathbf{u}}_{2h}^{n+1}\in X_h$ for $\forall \mathbf{v}_h\in X_h$ such that
\begin{align}
& (\frac{{\bar{\mathbf{u}}}_{1h}^{n+1} -\mathbf{u}_{h}^{n} }{\Delta t},\mathbf{v} ) + \frac{1}{Re}(\nabla\bar{\mathbf{u}}_{1h}^{n+1},\nabla\mathbf{v})  = (\nabla\cdot\mathbf{v},  p_{h}^{n}),\label{decoupled-numerical-NS1}\\
& (\frac{{\bar{\mathbf{u}}}_{2h}^{n+1}}{\Delta t},\mathbf{v} ) + \frac{1}{Re}(\nabla\bar{\mathbf{u}}_{2h}^{n+1},\nabla\mathbf{v}) = -(e^{{\bm{\psi}_h^{n+1}}},\nabla \mathbf{v})\nonumber\\
&- ((\mathbf{u}_{h}^{n}\cdot\nabla)\mathbf{u}_{h}^{n},\mathbf{v} )- (\sum\limits_{i}Coz_ic_{ih}^{n+1}\nabla \bar{V}_{h}^{n+1},\mathbf{v} ).\label{decoupled-numerical-NS2}
\end{align}
Step 5. We compute the auxiliary variable $r^{n+1}$ by
\begin{align}\label{solve-r}
\xi_h^{n+1} = \frac{ r^{n}  + \Delta t\zeta_{1h}}{\sqrt{E_{Ph}^{n+1}+ B} + \Delta t \zeta_{2h}},
\end{align}
where
\begin{align*}
\zeta_{1h} &= \frac{1}{2\sqrt{E_{Ph}^{n+1} + B}}\Big(\int_{K_h} \sum\limits_{i}Coz_i c_{ih}^{n+1}{\bar{\mathbf{u}}}_{1h}^{n+1} \cdot \nabla \bar{V}_{h}^{n+1} d\mathbf{x} \\
&\quad + \int_{K_h}(\mathbf{u}_{h}^{n}\cdot\nabla)\mathbf{u}_{h}^{n}\cdot {\bar{\mathbf{u}}}_{1h}^{n+1} d\mathbf{x} + \int_{K_h}\nabla \bar{\mathbf{u}}_{1h}^{n+1}:e^{{\bm{\psi}}_h^{n+1}} d\mathbf{x} \Big),\\
\zeta_{2h}& = \frac{1}{2\sqrt{E_{Ph}^{n+1} + B}}\Big(-\int_{K_h} \sum\limits_{i}Coz_i c_{ih}^{n+1}{\bar{\mathbf{u}}}_{2h}^{n+1} \cdot \nabla \bar{V}_{h}^{n+1} d\mathbf{x} - \int_{K_h}(\mathbf{u}_{h}^{n}\cdot\nabla)\mathbf{u}_{h}^{n}\cdot {\bar{\mathbf{u}}}_{2h}^{n+1} d\mathbf{x} \\
&\quad + \frac{Co}{Pe}\sum\limits_{i}\int_{K_h} c_{ih}^{n+1}|\nabla (\log c_{ih}^{n+1} + z_i\bar{V}_h^{n+1})|^2 d\mathbf{x} + \frac{M}{2Wi} \int_{K_h}\tr(e^{{\bm{\psi}}_{h}^{n+1}} + e^{-{\bm{\psi}}_{h}^{n+1}}- 2\mathbf{I})d\mathbf{x}\\
& \quad - \int_{K_h}\nabla\bar{ \mathbf{u}}_{2h}^{n+1}:e^{{\bm{\psi}}_{h}^{n+1}} d\mathbf{x}+ \frac{\kappa M}{2}\int_{K_h} \nabla \bm{\psi}_h^{n+1}:\nabla e^{\bm{\psi}_h^{n+1}} d\mathbf{x} \Big).
\end{align*}
We verify that \eqref{solve-r} is solvable by showing $\sqrt{E_{Ph}^{n+1}+ B} + \Delta t \zeta_2\neq 0$. By taking the $L^2$ inner product of \eqref{decoupled-numerical-NS2} with $\bar{\mathbf{u}}_{2}^{n+1}$, we have
\begin{align}
&- \int_{K_h}\nabla\bar{ \mathbf{u}}_{2h}^{n+1}:e^{{\bm{\psi}}_h^{n+1}} d\mathbf{x} - \int_{K_h}(\mathbf{u}_h^{n}\cdot\nabla)\mathbf{u}_h^{n}\cdot {\bar{\mathbf{u}}}_{2h}^{n+1} d\mathbf{x}
-\int_{K_h} \sum\limits_{i}Coz_i c_{ih}^{n+1}{\bar{\mathbf{u}}}_{2h}^{n+1} \cdot \nabla \bar{V}_h^{n+1} d\mathbf{x} \\
&= \frac{1}{\Delta t}\|{\bar{\mathbf{u}}}_{2h}^{n+1}\|^2 + \frac{1}{Re}\|\nabla\bar{\mathbf{u}}_{2h}^{n+1}\|^2\geq 0.
\end{align}
From \eqref{lemma1.2}, we can derive $\int_{K_h}\tr(e^{{\bm{\psi}}^{n+1}} + e^{-{\bm{\psi}}^{n+1}}- 2\mathbf{I})d\mathbf{x} \geq 0$ by choosing $\bm{\sigma}=e^{{\bm{\psi}}^{n+1}}$. By the positivity of concentration $c_i>0$, we have $\sum\limits_{i}\int_{K_h} c_{i}^{n+1}|\nabla (\log c_{ih}^{n+1} + z_i\bar{V}_h^{n+1})|^2 d\mathbf{x}\geq0$. Therefore,
\begin{align*}
\zeta_2 &= \frac{1}{2\sqrt{E_{P}^{n+1} + B}}\Big(\frac{Co}{Pe}\sum\limits_{i}\int_{K_h} c_{i}^{n+1}|\nabla g_{i}^{n+1}|^2 d\mathbf{x} + \frac{M}{2Wi} \int_{K_h}\tr(e^{{\bm{\psi}}^{n+1}} + e^{-{\bm{\psi}}^{n+1}}- 2\mathbf{I})d\mathbf{x}\\
&+ \frac{\kappa M}{2}\int_{K_h} \nabla \bm{\psi}^{n+1}:\nabla e^{\bm{\psi}^{n+1}} d\mathbf{x}\Big)+ \frac{1}{2\sqrt{E_{P}^{n+1} + B}}(\frac{1}{\Delta t}\|{\bar{\mathbf{u}}}_{2}^{n+1}\|^2 + \frac{1}{Re}\|\nabla\bar{\mathbf{u}}_{2}^{n+1}\|^2)\geq 0,
\end{align*}
which implies $\sqrt{E_{P}^{n+1}+ B} + \Delta t \zeta_2\neq 0$. Then we update $r^{n+1}$, $V_{h}^{n+1}$ and ${\bar{\mathbf{u}}}_{h}^{n+1}$ by
\begin{align*}
&r^{n+1} = \xi_h^{n+1}\sqrt{E_{Ph}^{n+1}+ B},\\
&V_{h}^{n+1} = \xi_h^{n+1}\bar{V}_{h}^{n+1},\\
&{\bar{\mathbf{u}}}_{h}^{n+1} = {\bar{\mathbf{u}}}_{1h}^{n+1}+ \xi^{n+1}{\bar{\mathbf{u}}}_{2h}^{n+1}.
\end{align*}
Step 6. Find $p_{h}^{n+1}\in X_h$ for $\forall q_{h}\in X_h$ such that
\begin{align*}
\left(\nabla (p_{h}^{n+1} - p_{h}^{n}),\nabla q_{h}\right)=\frac{1}{\Delta t}({\bar{\mathbf{u}}}_{h}^{n+1},\nabla q_{h}).
\end{align*}
Then we update ${\mathbf{u}}_{h}^{n+1}$ by
\begin{align*}
{\mathbf{u}}_{h}^{n+1} = {\bar{\mathbf{u}}}_{h}^{n+1} - \Delta t (\nabla p_{h}^{n+1} - \nabla p_{h}^{n}).
\end{align*}

\section{Numerical experiments}
In this section, we present several numerical examples to validate the proposed method. We first present the numerical results to examine the accuracy and stability of the scheme. We also investigate mass conservation and energy dissipation. Finally, we study the elastic effect of viscoelastic fluids on the electro-convection phenomena. 
All the numerical results are obtained by FreeFEM++ \cite{MR3043640}.

\subsection{Accuracy test}
The computational domain is assumed to be a rectangular region $\Omega = [0,1]^2$, and the terminal time is set to be $T = 1$. We choose the spatial step sizes $h_x = h_y =\sqrt{2}/120$ so that the spatial error is negligible compared with the temporal error. The parameters are $\lambda=1$ ,$z_p = 1$, $z_n = -1$, $Pe = 2$ , $Re = 1$, $Co = 5$, $M=1$, $\kappa=0.001$. 
The initial conditions are taken to be
\begin{align*}
&c_{p0} = 1.2 + \cos(\pi x)\cos(\pi y), \\
&c_{n0} = 1.2 - \cos(\pi x)\cos(\pi y), \\
&\bm{u}_0 = \bm{0},\quad \bm{\psi}_0 = \bm{I}.
\end{align*}
In our experiments, the reference solutions are replaced by the solution computed at a very fine mesh. Tables \ref{tab:first-L2accuracy-Wi-0.5}-\ref{tab:first-order-L2accuracy-Wi-10} illustrate the temporal convergence results for the first-order scheme for different Weissenberg numbers. We can observe that the expected convergence rates are achieved to confirm the accuracy of the proposed schemes.

\begin{table}[tbhp]
{\small
   \caption{Temporal convergence for the velocity $\bm{u}$, pressure $p$, concentrations $c_1$, $c_2$, electric potential $V$, tensor $\bm{\psi}$ and $r$ by using the $L^2$ norm with $Wi=0.5$. \label{tab:first-L2accuracy-Wi-0.5}}
  \begin{center}
     \begin{tabular}{llllllllll}
    \hline
       $\Delta t$      & $\|e_{\bm{u}}\|_2$   &order   &$\|e_{p}\|_2$    &order   &$\|e_{c_1}\|_2$   &order    &$\|e_{c_2}\|_2$   &order \\
    \hline
    $1/10$  & 0.000652369      &-      &   0.0939398        & -       &  0.000125439     &- &  0.000125413           &-
    \\
    $1/20$  &  0.000253842      &1.36   &  0.0468641      &  1.00   &  3.6066e-05       &1.80  & 3.60687e-05 & 1.80
    \\
    $1/40$  &  5.28484e-05     &2.26   &  0.0135061        &  1.79   &  1.18201e-05     &1.61 &1.18211e-05 & 1.61
    \\
    $1/80$ & 9.67292e-06      &2.45   & 0.00293073         &  2.20   &  3.93914e-06     &1.59 & 3.93918e-06 & 1.59
    \\
     \hline
     $\Delta t$       &$\|e_{V}\|_2$    &order  &$\|e_{\bm{\psi}}\|_2$     &order  &$\|e_{r}\|_2$   &order  \\
    \hline
    $1/10$
    &   1.2947e-05      & -      &     0.0691449     & -       &  0.0742016        &-\\
    $1/20$
    &   3.70181e-06   &  1.81   &    0.03275      & 1.08    & 0.0423519           & 0.81\\
    $1/40$
    &    1.20929e-06   & 1.61    &   0.0137475      & 1.25    & 0.0214446        & 0.98\\
    $1/80$
    &   4.02311e-07    & 1.59   &    0.00524009       & 1.39    &  0.00911348        & 1.23\\
    \hline
     \end{tabular}
  \end{center}
}
\end{table}

\begin{table}[tbhp]
{\small
   \caption{Temporal convergence for the velocity $\bm{u}$, pressure $p$, concentrations $c_1$, $c_2$, electric potential $V$, tensor $\bm{\psi}$ and $r$ by using the $L^2$ norm with $Wi=5$. \label{tab:L2accuracy-Wi-5}}
  \begin{center}
     \begin{tabular}{llllllllll}
    \hline
       $\Delta t$      & $\|e_{\bm{u}}\|_2$   &order   &$\|e_{p}\|_2$    &order   &$\|e_{c_1}\|_2$   &order    &$\|e_{c_2}\|_2$   &order \\
    \hline
    $1/10$  & 0.0001947      &-      & 0.103931      & -      & 0.000125807     &-  &0.000125802  &-
    \\
    $1/20$  & 5.93838e-05    &1.71   & 0.0578444     & 0.85   & 3.61298e-05    &1.80 &3.61264e-05 &1.80
    \\
    $1/40$  & 1.70225e-05    &1.80   & 0.0206123     &  1.49   & 1.18335e-05    &1.61 & 1.18344e-05 &1.61
    \\
    $1/80$ &  3.89156e-06    &2.13   & 0.00540877    &  1.93   & 3.94263e-06    &1.59 &3.94261e-06  &1.59
    \\
     \hline
     $\Delta t$       &$\|e_{V}\|_2$    &order  &$\|e_{\bm{\psi}}\|_2$     &order  &$\|e_{r}\|_2$   &order  \\
    \hline
    $1/10$
    &  1.29429e-05       & -      &   0.0446843       & -       & 0.0618004        &-\\
    $1/20$
    &   3.70085e-06      & 1.81   & 0.0263833           & 0.76    &   0.0364599        & 0.76\\
    $1/40$
    &  1.20905e-06       & 1.61   &  0.00973823         & 1.44    &  0.0187784        & 0.96\\
    $1/80$
    &  4.02246e-07       & 1.59   &  0.00260591          & 1.90    &   0.00805101       & 1.22\\
    \hline
     \end{tabular}
  \end{center}
}
\end{table}

\begin{table}[tbhp]
{\small
   \caption{Temporal convergence for the velocity $\bm{u}$, pressure $p$, concentrations $c_1$, $c_2$, electric potential $V$, tensor $\bm{\psi}$ and $r$ by using the $L^2$ norm with $Wi=10$. \label{tab:first-order-L2accuracy-Wi-10}}
  \begin{center}
     \begin{tabular}{llllllllll}
    \hline
       $\Delta t$      & $\|e_{\bm{u}}\|_2$   &order   &$\|e_{p}\|_2$    &order   &$\|e_{c_1}\|_2$   &order    &$\|e_{c_2}\|_2$   &order \\
    \hline
    $1/10$  & 0.000160325  &-      &    0.103804        & -       &    0.00012581         &-   &0.000125807  &-
    \\
    $1/20$  & 6.43392e-05  &1.32   &  0.0587344      &  0.82   &     3.61305e-05       &1.80 &3.61269e-05  &1.80
    \\
    $1/40$  & 2.45813e-05  &1.39   &   0.0213994      &  1.46   &     1.18338e-05      &1.61  & 1.18344e-05  &1.61
    \\
    $1/80$  & 5.63447e-06  &2.13   &  0.00569618      &  1.91   &    3.94266e-06      &1.59  & 3.94267e-06 &1.59
    \\
     \hline
     $\Delta t$       &$\|e_{V}\|_2$    &order  &$\|e_{\bm{\psi}}\|_2$     &order  &$\|e_{r}\|_2$   &order  \\
    \hline
    $1/10$
    &   1.29426e-05      & -      &   0.044515        & -       &   0.0617406      &-\\
    $1/20$
    &   3.7008e-06      & 1.81   & 0.0267973        & 0.73    &  0.0364467        & 0.76\\
    $1/40$
    &   1.20904e-06      & 1.61   &  0.010093       & 1.41    &  0.0187763       & 0.96\\
    $1/80$
    &   4.02243e-07      & 1.59   &  0.0027305       & 1.89    &  0.00805112       & 1.22\\
    \hline
     \end{tabular}
  \end{center}
}
\end{table}

\subsection{The energy dissipation law}
The parameters used in the calculation are: $T=5$, $z_p = 1$, $z_n = -1$, $\lambda=0.1$ , $M=1$, $Pe = 40$ , $Re = 0.5$, $Co = 1$, $\kappa=0.01$, $h_x = h_y =\sqrt{2}/140$. This example aims to verify the mass conserving and energy dissipation property. 
We test the numerical scheme with the initial conditions:
\begin{align*}
&c_{p0} = 1.1 + \cos(\pi x)\cos(\pi y), \\
&c_{n0} = 1.1 - \cos(\pi x)\cos(\pi y), \\
&\bm{u}_0 = \bm{0},\\
&\bm{\psi}_0 = \left(
                       \begin{array}{cc}
                         2(x^2(1-x)^2+y^2(1-y)^2) + 0.1 & 0 \\
                         0 & x^2(1-x)^2+y^2(1-y)^2 + 0.1 \\
                       \end{array}
                     \right).
\end{align*}
We present the time evolution of the discrete total energy for different time steps when $Wi=1$ in Fig. \ref{EnergyDecay} and different $Wi$ when 
$\Delta t=0.001$ in Fig. \ref{EnergyDecayWi}. Fig. \ref{MassConservation} plots time evolution of the discrete masses $\int_{\Omega} c_i d\mathbf{x}, i = p, n$, which demonstrates the mass conservation property of the numerical scheme. 
\begin{figure}[htb]
\setlength{\abovecaptionskip}{0pt}
\setlength{\belowcaptionskip}{0pt}
\renewcommand*{\figurename}{Fig.}
\centering
\begin{minipage}[t]{0.41\linewidth}
\includegraphics[width=2.15in, height=2.2 in]{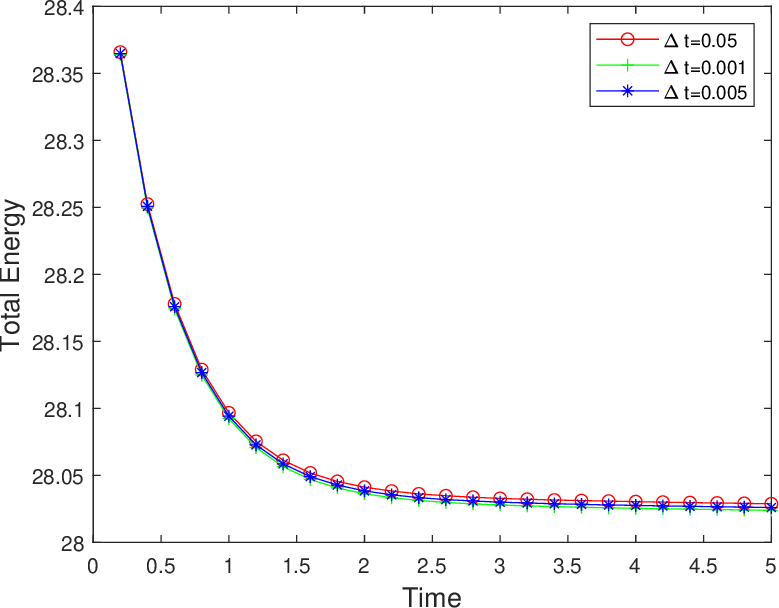}
\end{minipage}%
\caption{ Time evolution of the discrete total energy for different time step.\label{EnergyDecay}}
\end{figure}

\begin{figure}[htb]
\setlength{\abovecaptionskip}{0pt}
\setlength{\belowcaptionskip}{0pt}
\renewcommand*{\figurename}{Fig.}
\centering
\begin{minipage}[t]{0.41\linewidth}
\includegraphics[width=2.15in, height=2.2 in]{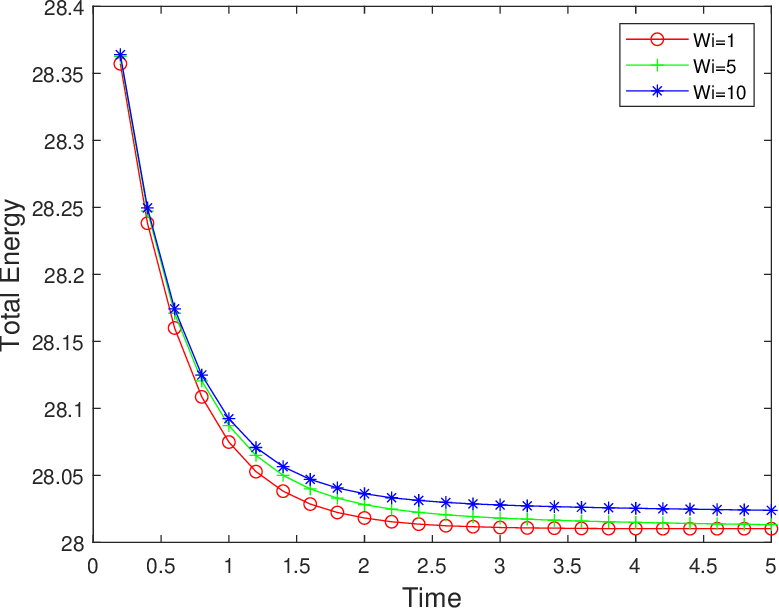}
\end{minipage}%
\caption{ Time evolution of the discrete total energy for  different Wi.\label{EnergyDecayWi}}
\end{figure}

\begin{figure}[htbp]
\setlength{\abovecaptionskip}{1pt}
\renewcommand*{\figurename}{Fig.}
\centering
\begin{minipage}[t]{0.41\linewidth}
\includegraphics[width=2.15in, height=1.8in]{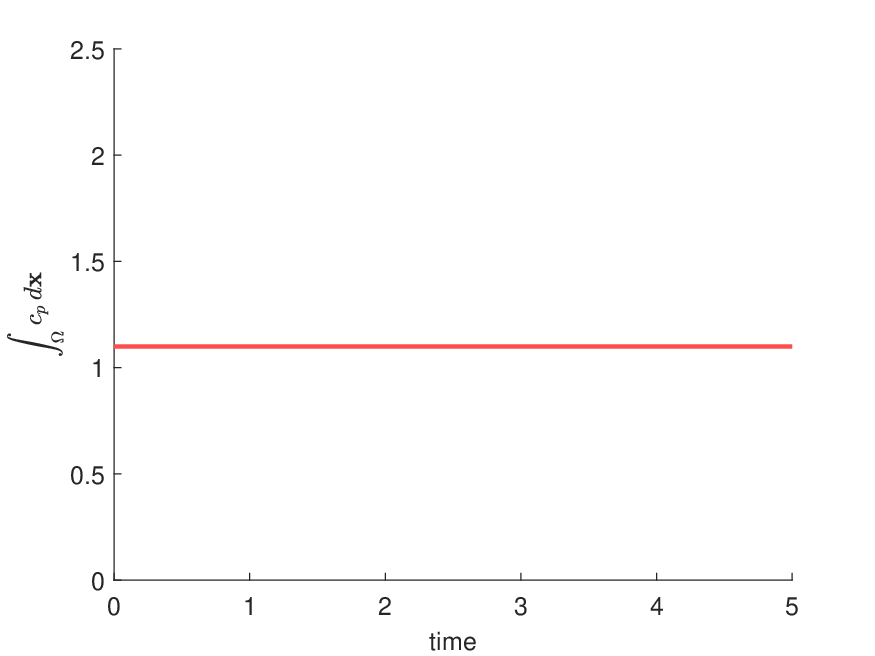}
\end{minipage}
\begin{minipage}[t]{0.41\linewidth}
\includegraphics[width=2.15in, height=1.8in]{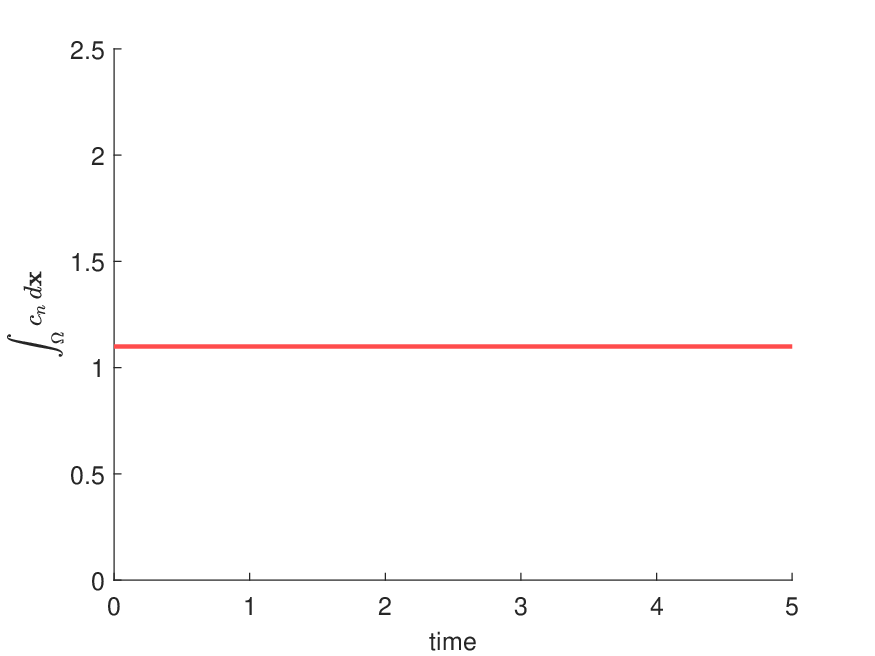}
\end{minipage}
\caption{Time evolution of the discrete mass $\int_{\Omega} c_i d\mathbf{x}, i = p, n$.\label{MassConservation}}
\end{figure}

\subsection{Dynamics with initial discontinuous concentrations}
We investigate the dynamics of the system on the unit square with discontinuous concentrations which represents an interface between the electrolyte and the solid surfaces where electroosmosis (transport of ions from the electrolyte towards the solid surface) occurs. 
The initial conditions are given as follows
\begin{align*}
&c_p=
\begin{cases}
1,(0,1)^2\backslash{(0,0.75)\times(0,1)\cup(0.75,1)\times(0,\frac{11}{20})},\\
0.2,\text{otherwise}
\end{cases}\\
&c_n=
\begin{cases}
1,(0,1)^2\backslash{(0,0.75)\times(0,1)\cup(0.75,1)\times(\frac{9}{20},1)},\\
0.2,\text{otherwise}
\end{cases}\\
& \mathbf{u}_0=\mathbf{0},\quad \bm{\psi}_0 = \bm{I}.
\end{align*}

\begin{figure}[htbp]
\setlength{\abovecaptionskip}{1pt}
\setlength{\abovecaptionskip}{1pt}
\setlength{\belowcaptionskip}{1pt}
\renewcommand*{\figurename}{Fig.}
\centering
\begin{minipage}[t]{0.3\linewidth}
\includegraphics[width=1.80in, height=1.4in]{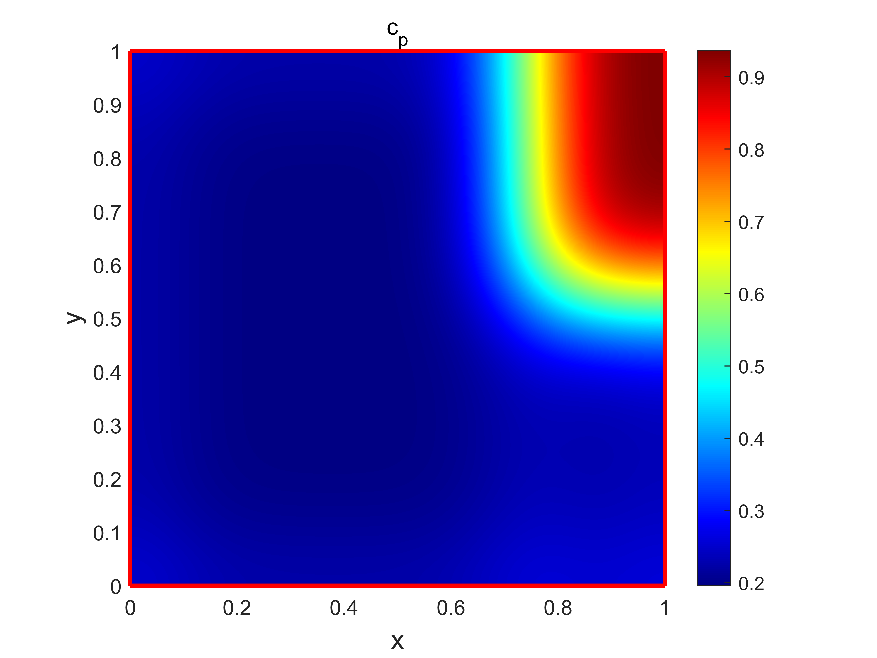}
\end{minipage}
\begin{minipage}[t]{0.3\linewidth}
\includegraphics[width=1.80in, height=1.4in]{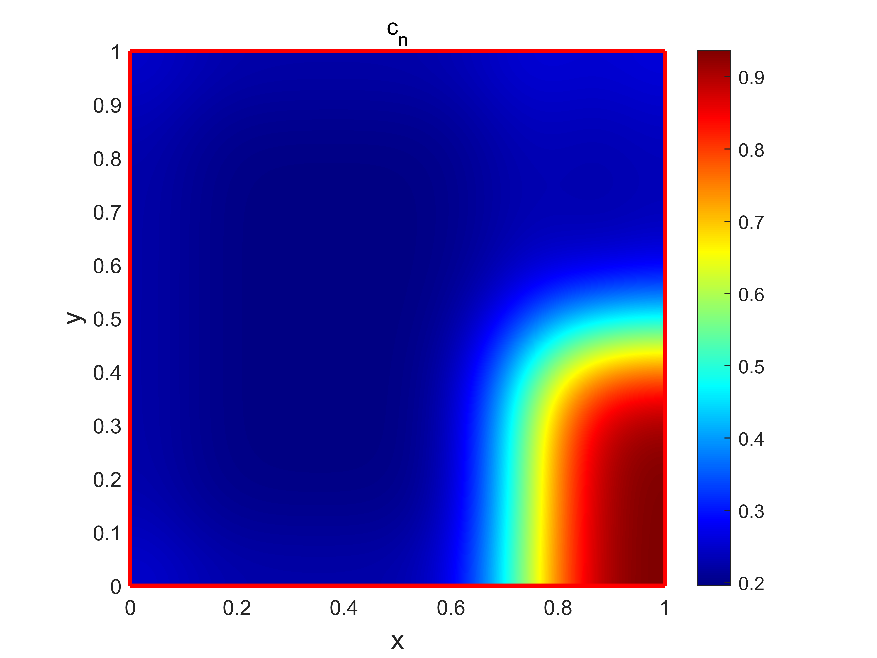}
\end{minipage}
\begin{minipage}[t]{0.3\linewidth}
\includegraphics[width=1.80in, height=1.4in]{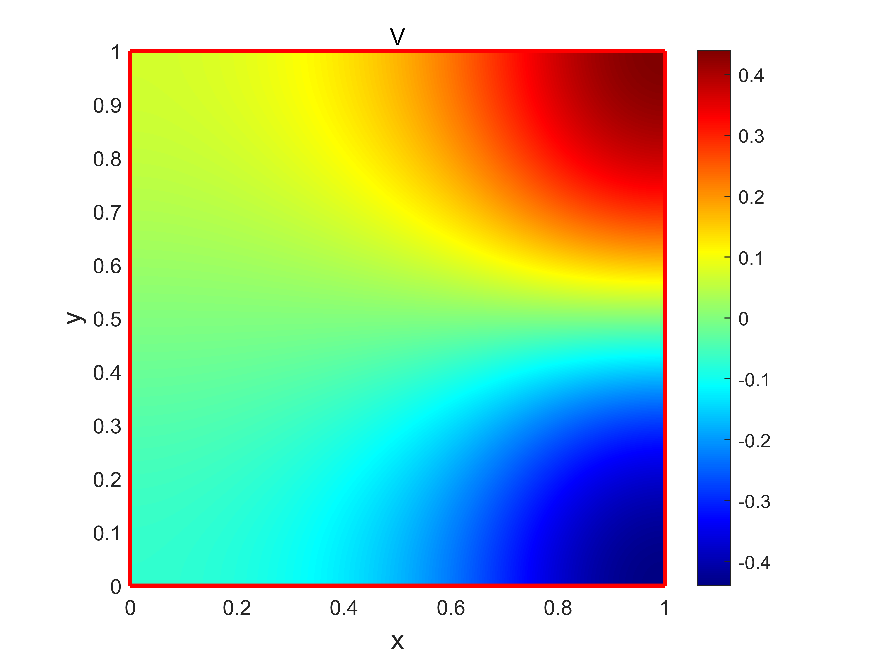}
\end{minipage}
\\
\begin{minipage}[t]{0.3\linewidth}
\includegraphics[width=1.80in, height=1.4in]{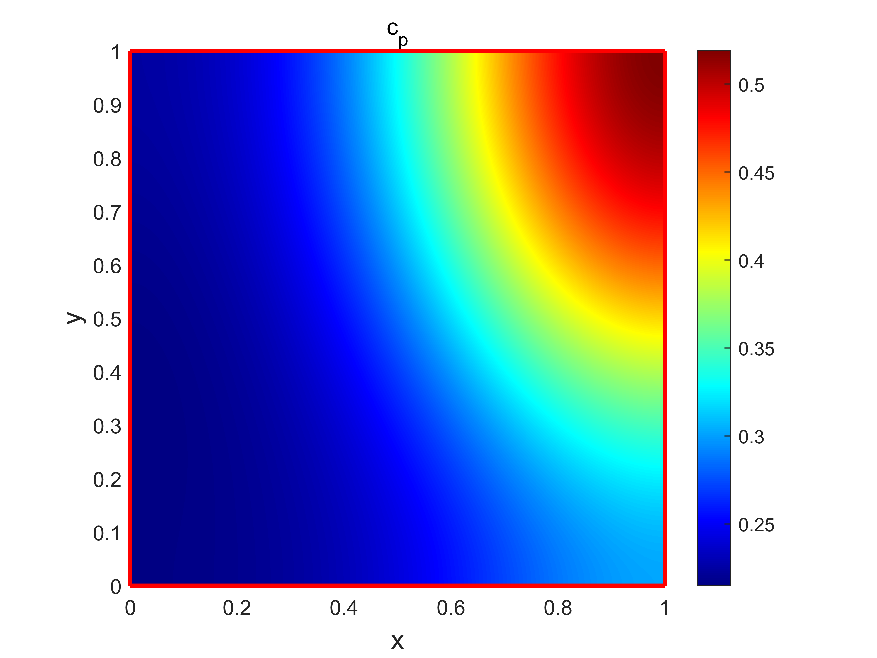}
\end{minipage}
\begin{minipage}[t]{0.3\linewidth}
\includegraphics[width=1.80in, height=1.4in]{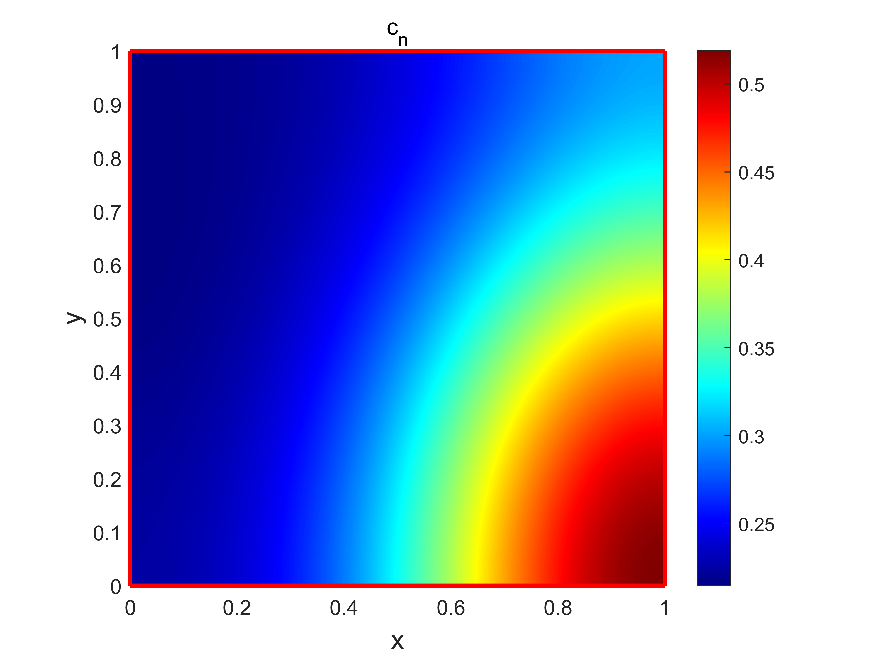}
\end{minipage}
\begin{minipage}[t]{0.3\linewidth}
\includegraphics[width=1.80in, height=1.4in]{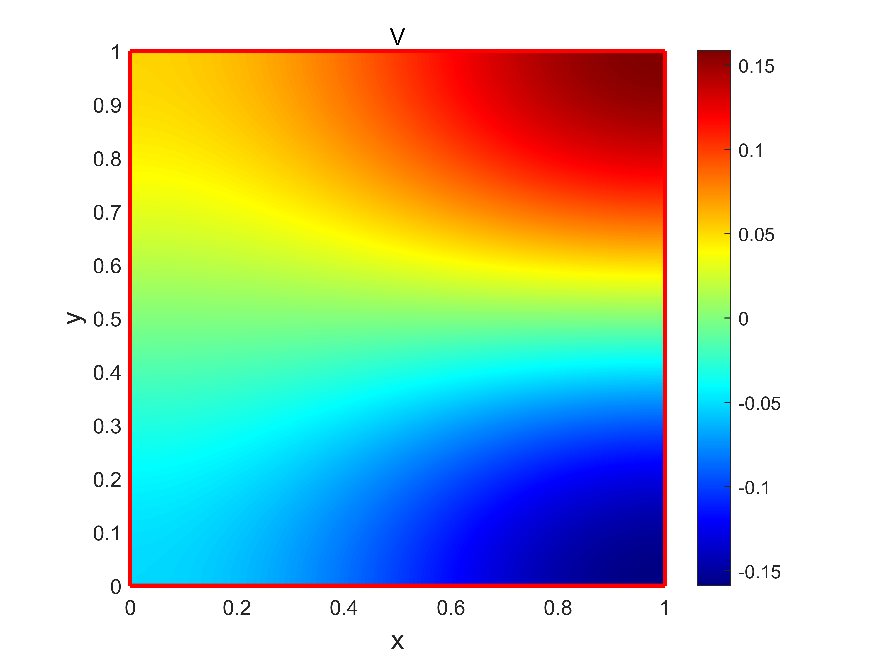}
\end{minipage}
\\
\begin{minipage}[t]{0.3\linewidth}
\includegraphics[width=1.80in, height=1.4in]{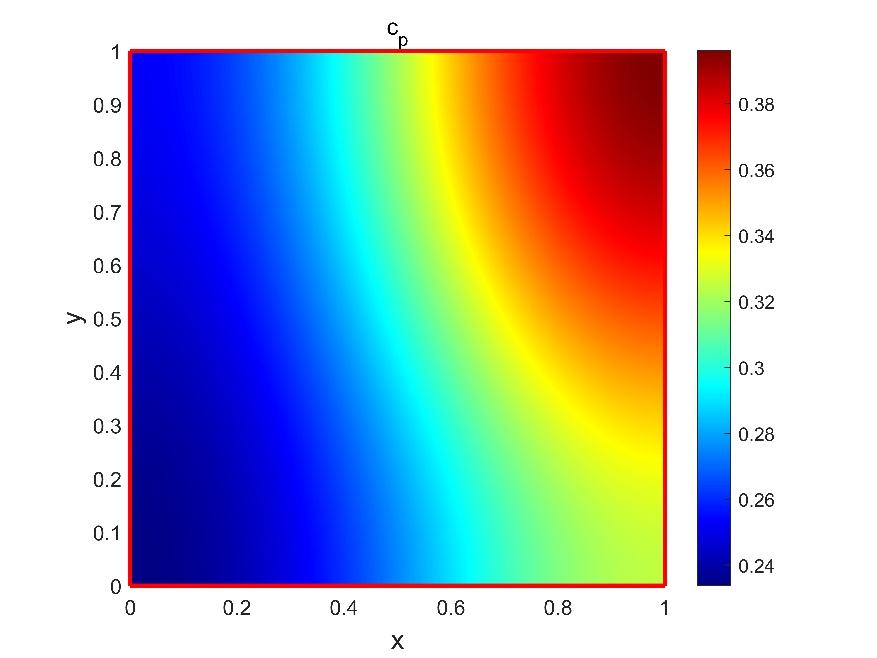}
\end{minipage}
\begin{minipage}[t]{0.3\linewidth}
\includegraphics[width=1.80in, height=1.4in]{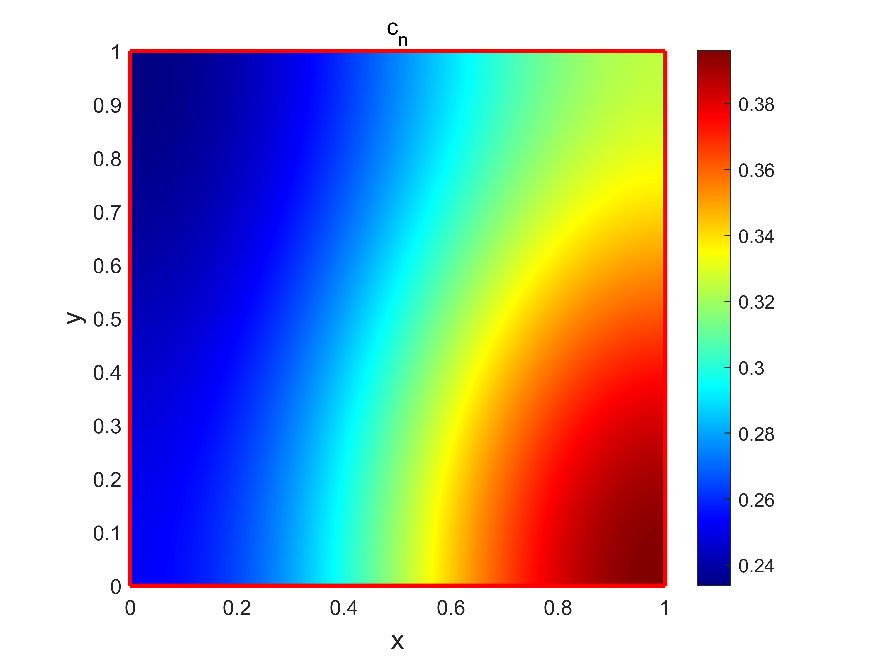}
\end{minipage}
\begin{minipage}[t]{0.3\linewidth}
\includegraphics[width=1.80in, height=1.4in]{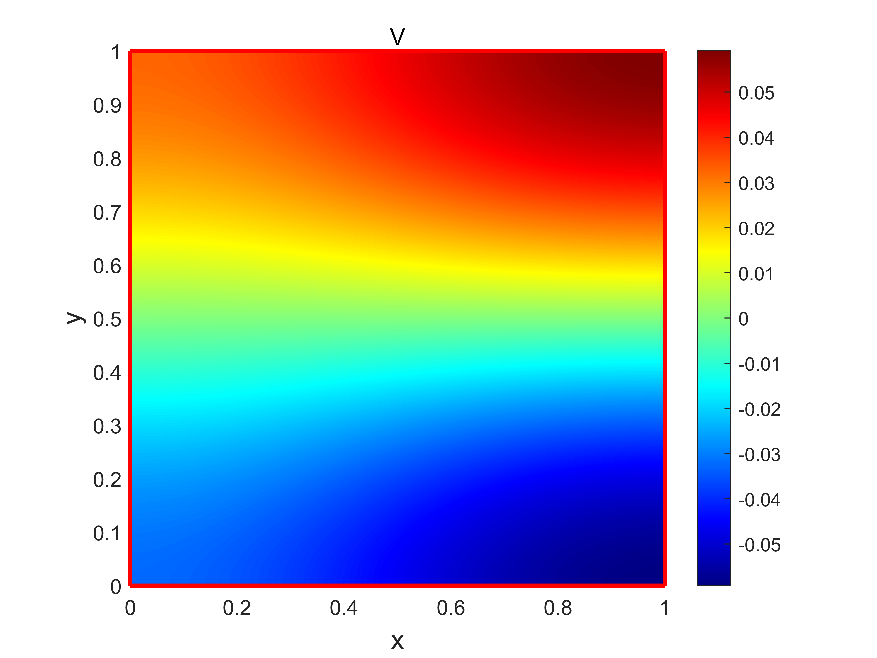}
\end{minipage}
\caption{Snapshots of the approximate solutions $(c_p,c_n)$ at $t = 0.1$ (top row), $t = 1.0$ (middle) and $t = 2.0$ (bottom).\label{Discontinuousconcentrations}}
\end{figure}

\begin{figure}[htbp]
\setlength{\abovecaptionskip}{1pt}
\setlength{\abovecaptionskip}{1pt}
\setlength{\belowcaptionskip}{1pt}
\renewcommand*{\figurename}{Fig.}
\centering
\begin{minipage}[t]{0.3\linewidth}
\includegraphics[width=1.80in, height=1.4in]{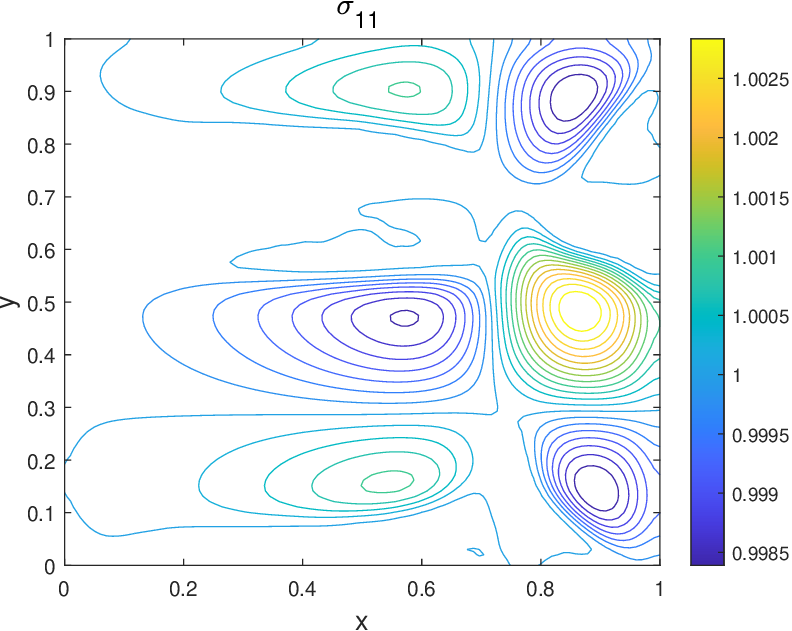}
\end{minipage}
\begin{minipage}[t]{0.3\linewidth}
\includegraphics[width=1.80in, height=1.4in]{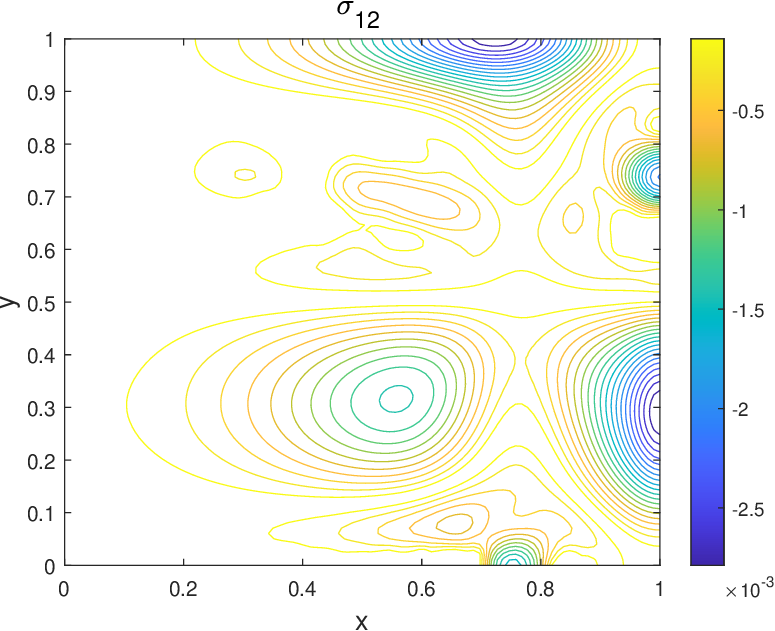}
\end{minipage}
\begin{minipage}[t]{0.3\linewidth}
\includegraphics[width=1.80in, height=1.4in]{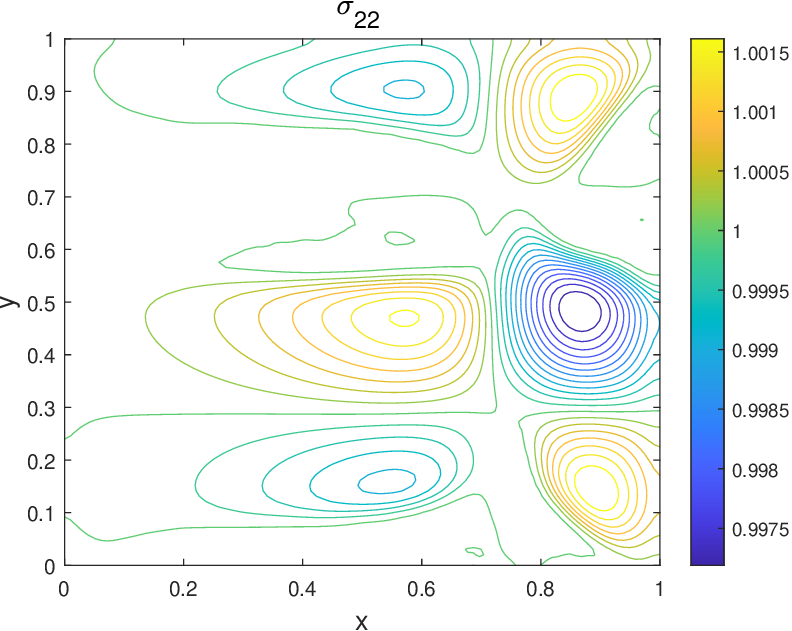}
\end{minipage}
\\
\begin{minipage}[t]{0.3\linewidth}
\includegraphics[width=1.80in, height=1.4in]{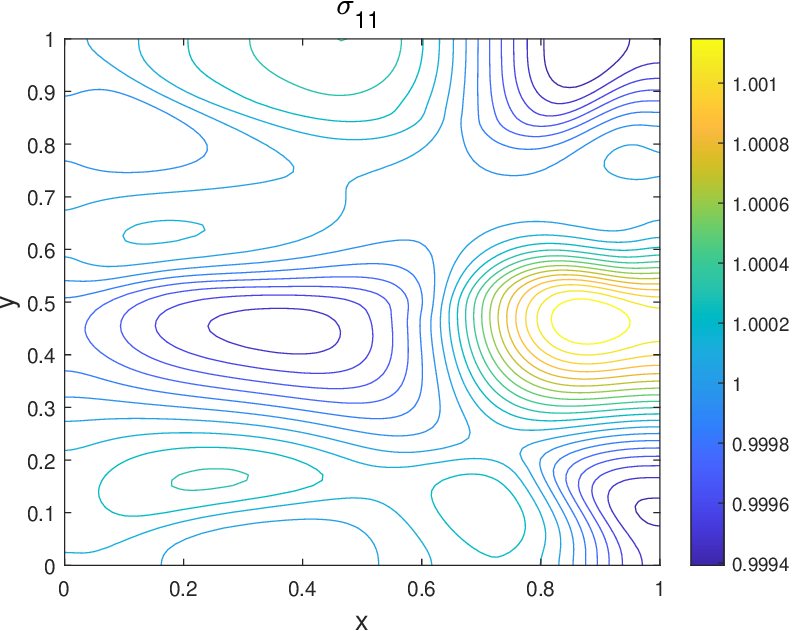}
\end{minipage}
\begin{minipage}[t]{0.3\linewidth}
\includegraphics[width=1.80in, height=1.4in]{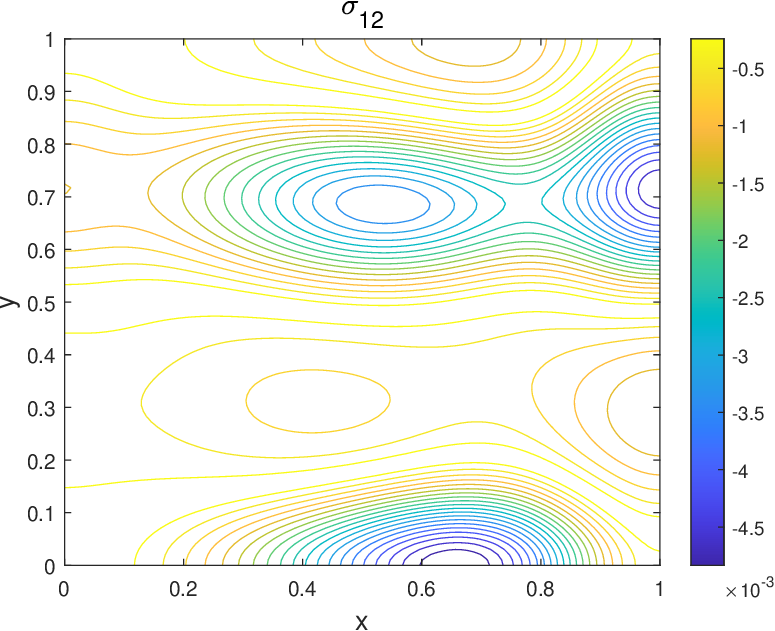}
\end{minipage}
\begin{minipage}[t]{0.3\linewidth}
\includegraphics[width=1.80in, height=1.4in]{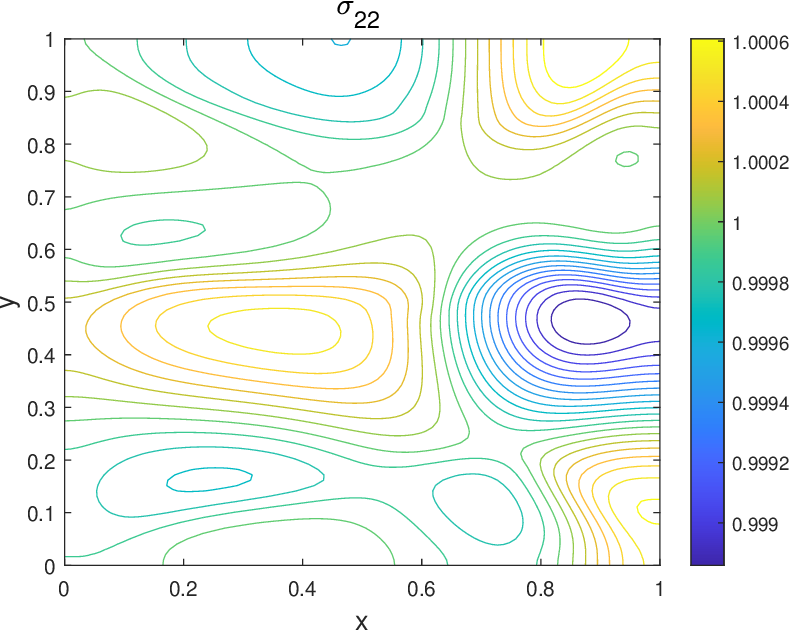}
\end{minipage}
\\
\begin{minipage}[t]{0.3\linewidth}
\includegraphics[width=1.80in, height=1.4in]{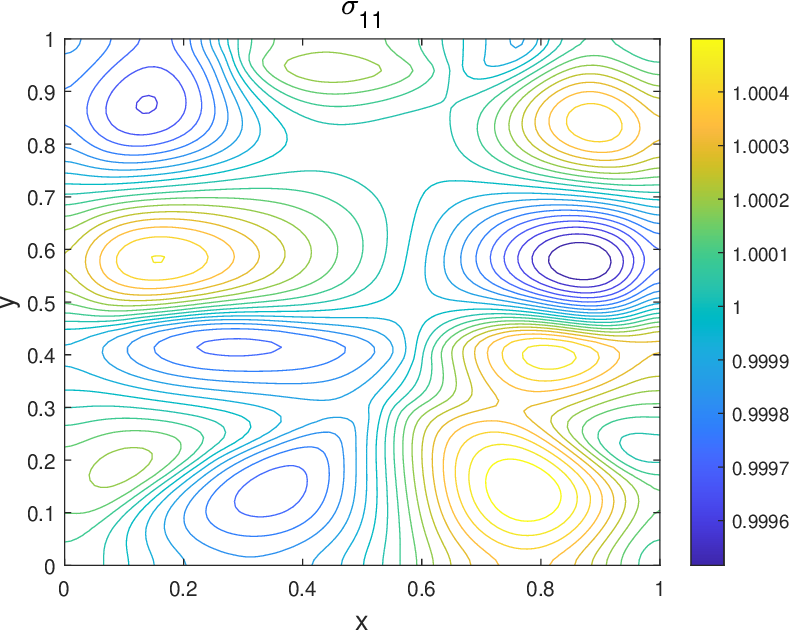}
\end{minipage}
\begin{minipage}[t]{0.3\linewidth}
\includegraphics[width=1.80in, height=1.4in]{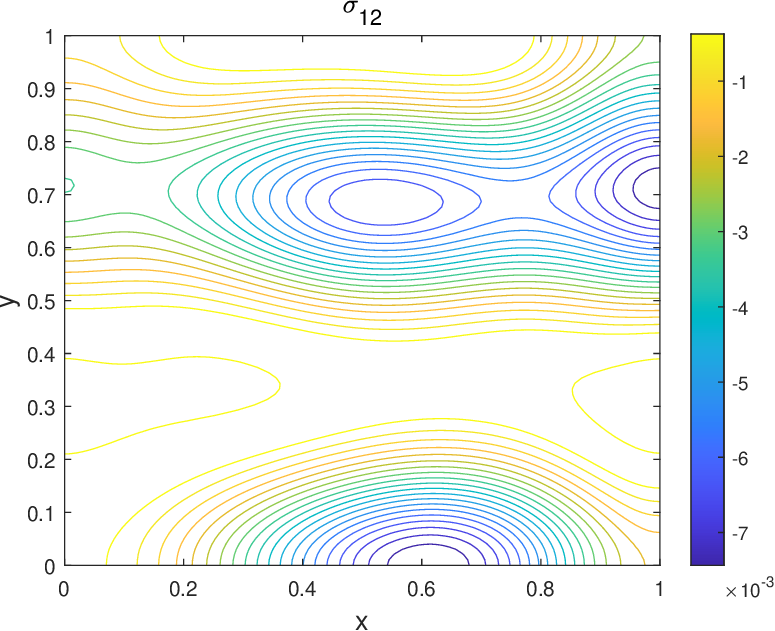}
\end{minipage}
\begin{minipage}[t]{0.3\linewidth}
\includegraphics[width=1.80in, height=1.4in]{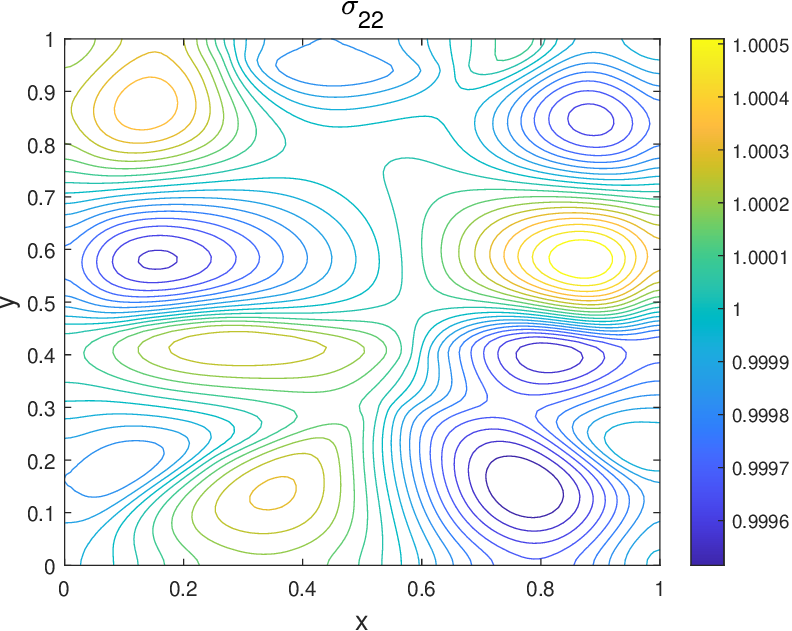}
\end{minipage}
\caption{Snapshots of the tensor $\sigma_{11}$, $\sigma_{12}$, $\sigma_{22}$ for $t=0.1,1,2$(from left to right).\label{Discontinuousconcentrationss}}
\end{figure}
We set $\lambda=0.1$,  $z_p = 1$, $z_n = -1$, $Pe=20$, $Re=10$, $Co=2.0$, $Wi=1$, $\kappa=0.001$, the time step $\Delta t = 1.0e-3$, the mesh size $h=\sqrt{2}/64$, 
The subsequent snapshots of two concentrations at $t = 0.1$, $t = 1.0$ and $t = 2.0 $ are depicted in Fig. \ref{Discontinuousconcentrations}, which match the results in \cite{gao2017linearized, prohl2009convergent, li2023error}.
The subsequent snapshots of the tensor $\sigma_{11}$, $\sigma_{12}$, $\sigma_{22}$ are also depicted in Fig. \ref{Discontinuousconcentrationss}.
\begin{figure}[htb]
\setlength{\abovecaptionskip}{1pt}
\setlength{\belowcaptionskip}{1pt}
\renewcommand*{\figurename}{Fig.}
\centering
\begin{minipage}[t]{0.41\linewidth}
\includegraphics[width=2.5in, height=2.2 in]{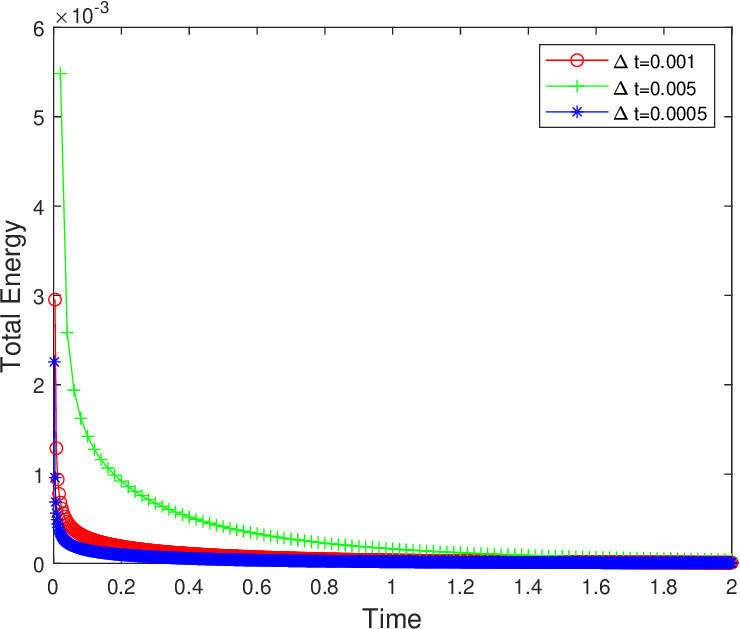}
\end{minipage}%
\caption{ Time evolution of the discrete total energy.\label{DEnergyDecay}}
\end{figure}

\subsection{Lid-driven cavity problem}
We consider one of the standard test problems for viscoelastic fluid flows: lid-driven cavity problem.
The results for the Oldroyd-B viscoelastic fluid flows without charge injection have been reported in \cite{venkatesan2017three}.
Unlike Newtonian fluids, viscoelastic fluids cannot sustain deformations near the upper corners and therefore the motion of the lid needs to be regularized such that $\nabla \bm{u}$ vanishes at the corners \cite{fattal2005time}.
We set the boundary conditions to be
\begin{equation}\label{test3:InitBC}
\begin{split}
&x=0:\quad c_p =1, \quad c_n=1,  \quad V=0, \quad u=0, \quad v=0,\quad\frac{\partial \bm{\sigma}}{\partial \mathbf{n}}=0, \\
&x=1:\quad c_p =11, \quad c_n=1, \quad V=1,\quad u=0, \quad v=0, \quad\frac{\partial \bm{\sigma}}{\partial \mathbf{n}}=0,\\
&y=0 :\quad \frac{\partial c_i}{\partial \mathbf{n}} =0(i=p,n), \quad \frac{\partial V}{\partial \mathbf{n}} =0, \quad u=0, \quad v=0,\quad \frac{\partial \bm{\sigma}}{\partial \mathbf{n}}=0,\\
&y=1 : \quad  \frac{\partial c_i}{\partial \mathbf{n}} =0(i=p,n), \quad\frac{\partial V}{\partial \mathbf{n}} =0, \quad u=16[1+\tanh(8(t-0.5))]x^2(1-x)^2, \quad v=0, \quad\frac{\partial \bm{\sigma}}{\partial \mathbf{n}}=0.
\end{split}
\end{equation}
The initial conditions are
\begin{align*}
&\bm{u}_0=\bm{0},\quad c_{p0}=10x+1,\quad c_{n0}=1, \quad \bm{\psi}=\bm{0}.
\end{align*}
The parameters are set $z_p = 1$, $z_n = -1$, $\lambda=0.5$, $Pe=10$, $Re=1$, $\kappa=0.001$, $M=1$. The time step is $\Delta t = 1.0e-3$. the mesh size is $h_x=h_y=\sqrt{2}/128$, and the final time $T = 5$. Fig.\ref{cross-section-tensor} shows the solution of $\sigma_{11}$ and $\sigma_{22}$, respectively, along the cross section line $y=1$. We can observe that the maximum value of $\sigma_{11}$ and $\sigma_{22}$ increases significantly when increasing the Weissenberg number.
The numerical solution of the horizontal velocity component $u$ along the cross-section line $x=0.5$ and vertical velocity component $v$ along the cross-section line $y=0.75$ are shown in Fig. \ref{cross-section-velocity}. As the Weissenberg number increases, the minimum value of the horizontal velocity component decreases in magnitude and its location moves closer to the lid. The extreme values of vertical velocity component decreases with an increase in the Weissenberg number. The contour plots of the components of the viscoelastic conformation stress $\sigma_{11}$, $\sigma_{12}$ and $\sigma_{22}$ with $Co=0.1$ are depicted in Fig. \ref{Contour-tensor} . We can observe that $\sigma_{11}$ has a thin boundary layer along the lid, whereas $\sigma_{12}$ and $\sigma_{22}$ have high gradient near the upper downstream corner.
Fig. \ref{Lid-driven-Streamlines} presents the flow field at $Co=0.1, M=1$ with $Wi = 1, 3, 20$, at $Wi=1, M=1$ with $Co = 0.1, 30, 100$ and at $Wi=3, C0=2$ with $M = 0, 1, 10$, respectively. It is well-known that for a Newtonian fluid when $M=0$, the lid-driven cavity flow problem can lead to a symmetrical horizontal location of the vortex. However, due to the presence of elastic effects in viscoelastic fluid, this symmetry is broken.
To be more specific, as the Weissenberg number increases, the large normal stresses that are generated in the viscoelastic fluid are advected into the downstream direction, leading to an increase in the flow resistance. To compensate this effect, the vortex shifts to the left. The experimental observations are consistent with literatures \cite{fattal2005time,venkatesan2017three}. A leftward shift of the vortex is also observed with an increase in $M$ number. 
Similarly, as the Coulomb-driven number increases, the stronger Coulomb force on the left attracts charged particles toward the left, causing the vortex location of the viscoelastic fluid to move toward the left direction.  

\begin{figure}[htbp]
\setlength{\abovecaptionskip}{1pt}
\setlength{\abovecaptionskip}{1pt}
\setlength{\belowcaptionskip}{1pt}
\renewcommand*{\figurename}{Fig.}
\centering
\begin{minipage}[t]{0.3\linewidth}
\includegraphics[width=1.80in, height=1.4in]{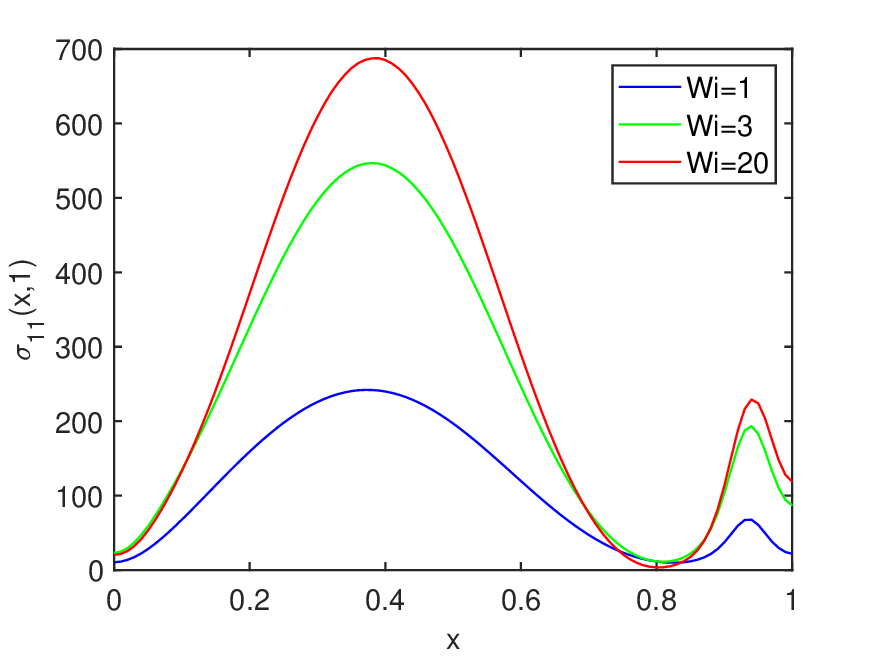}
\end{minipage}
\begin{minipage}[t]{0.3\linewidth}
\includegraphics[width=1.80in, height=1.4in]{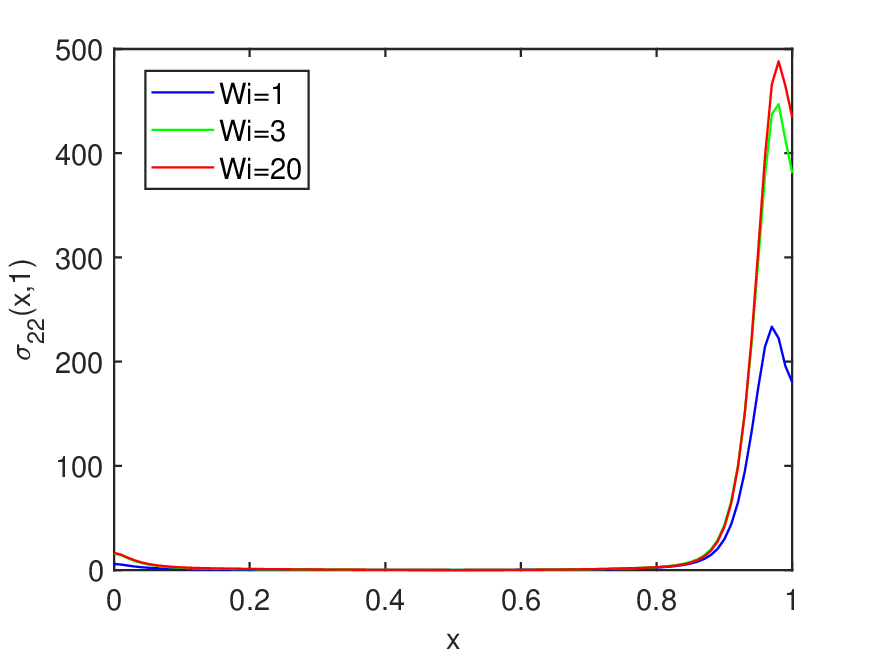}
\end{minipage}
\\
\begin{minipage}[t]{0.3\linewidth}
\includegraphics[width=1.80in, height=1.4in]{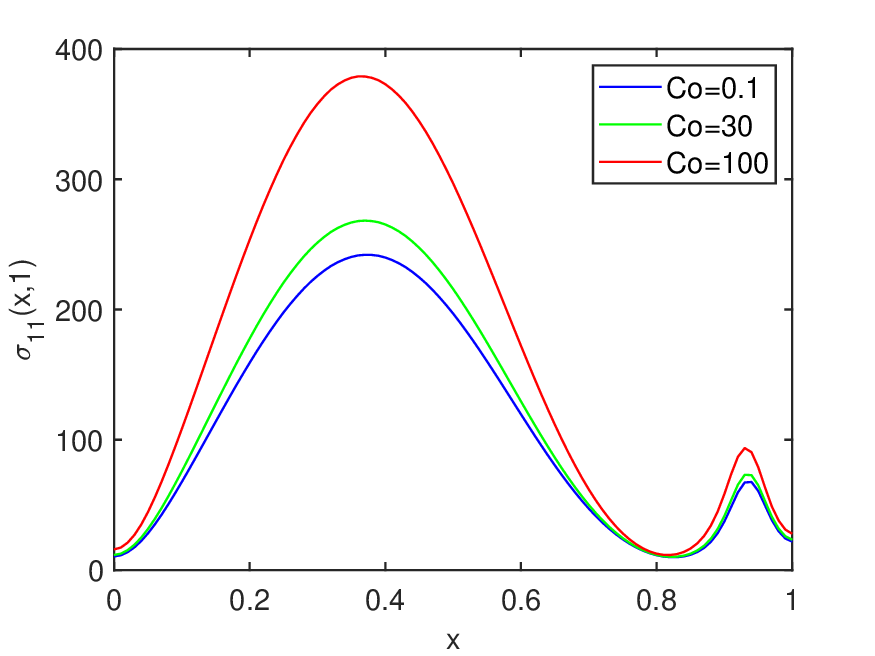}
\end{minipage}
\begin{minipage}[t]{0.3\linewidth}
\includegraphics[width=1.80in, height=1.4in]{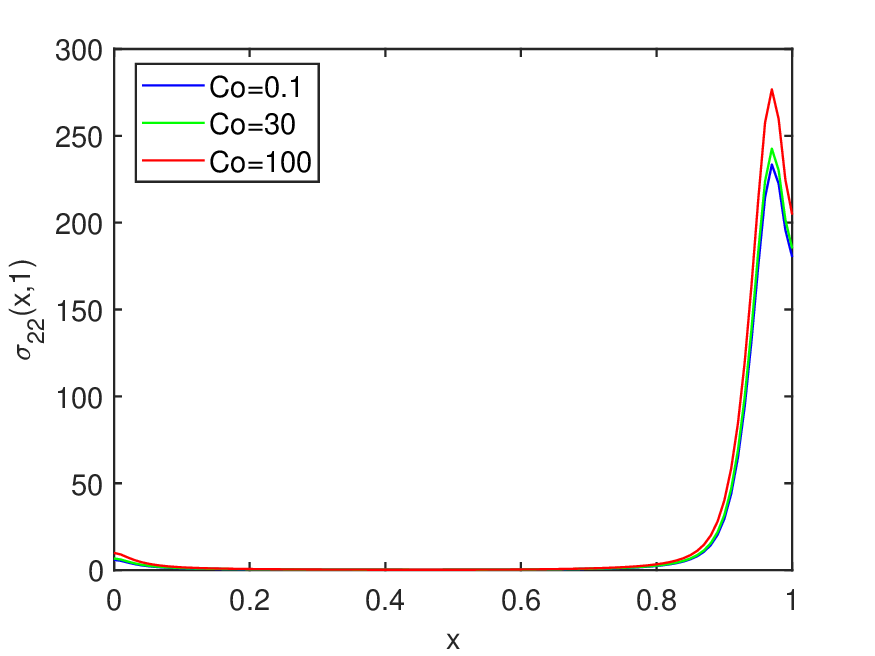}
\end{minipage}
\caption{The cross-section of $\sigma_{11}$ and $\sigma_{22}$ for different Wi and Co.\label{cross-section-tensor}}
\end{figure}

\begin{figure}[htbp]
\setlength{\abovecaptionskip}{1pt}
\setlength{\abovecaptionskip}{1pt}
\setlength{\belowcaptionskip}{1pt}
\renewcommand*{\figurename}{Fig.}
\centering
\begin{minipage}[t]{0.3\linewidth}
\includegraphics[width=1.80in, height=1.4in]{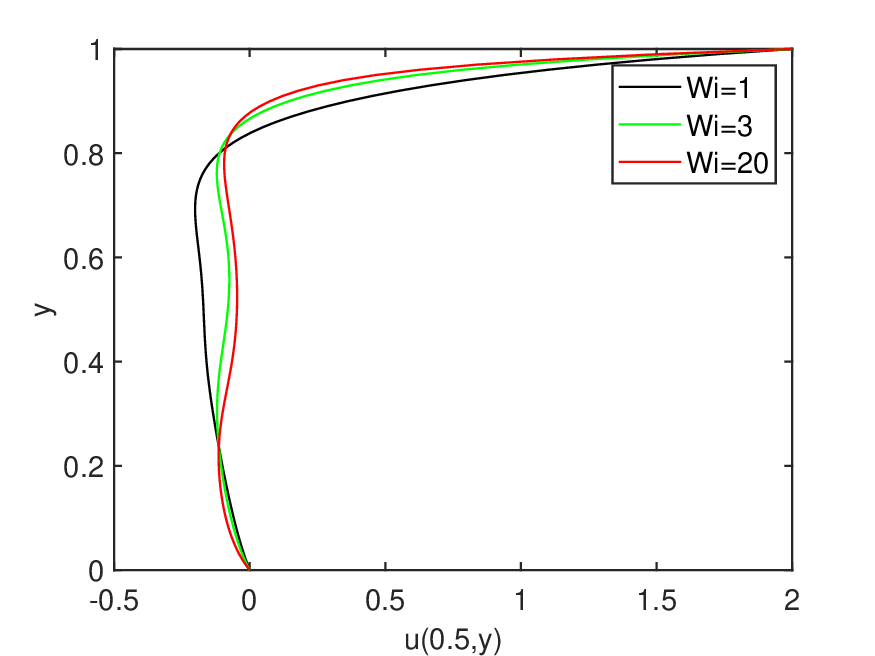}
\end{minipage}
\begin{minipage}[t]{0.3\linewidth}
\includegraphics[width=1.80in, height=1.4in]{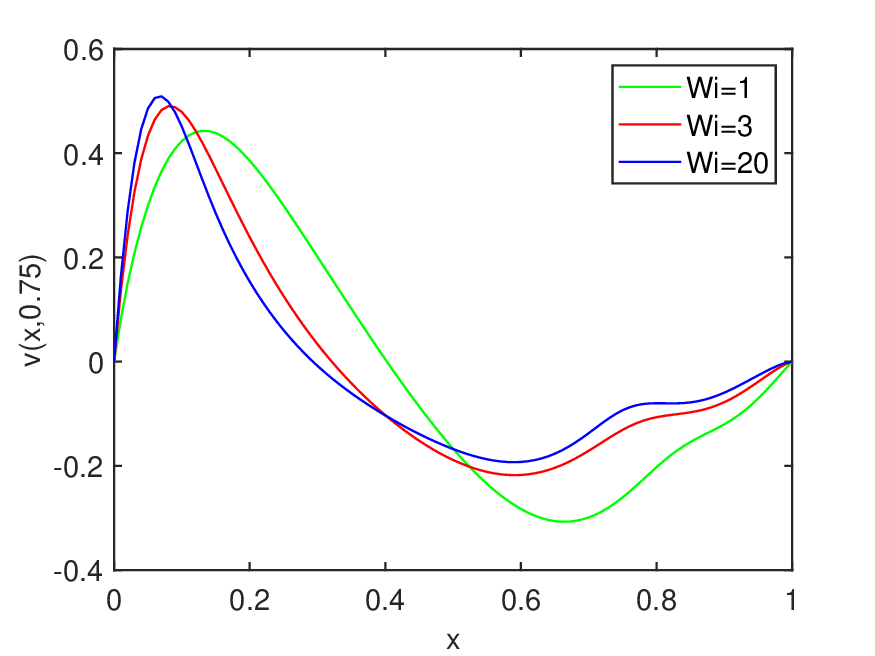}
\end{minipage}
\\
\begin{minipage}[t]{0.3\linewidth}
\includegraphics[width=1.80in, height=1.4in]{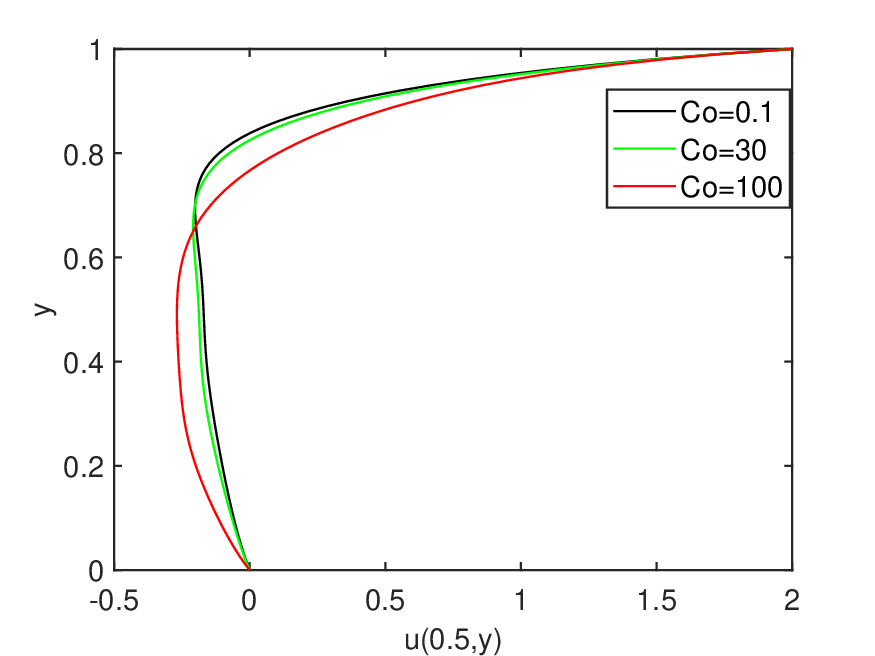}
\end{minipage}
\begin{minipage}[t]{0.3\linewidth}
\includegraphics[width=1.80in, height=1.4in]{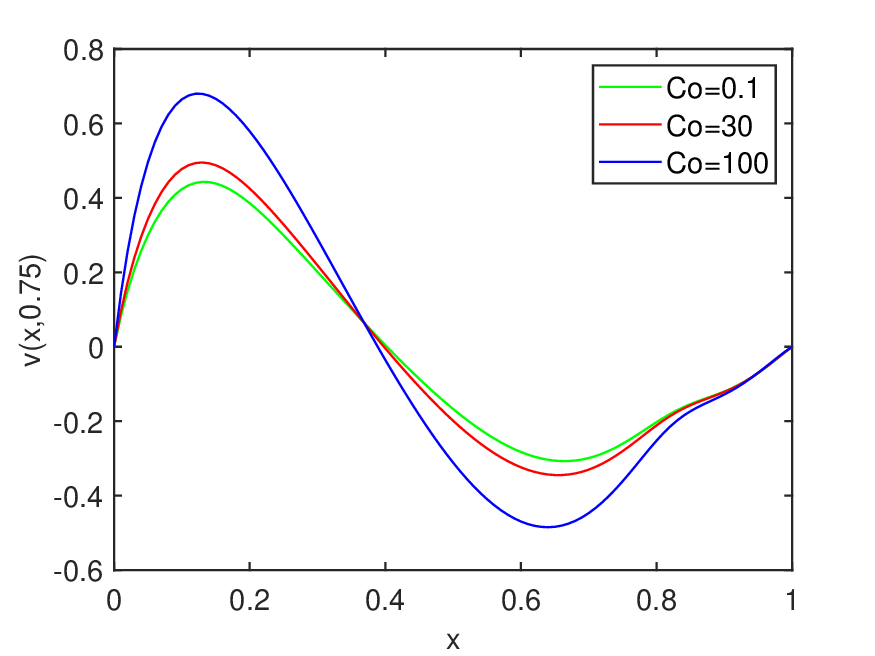}
\end{minipage}
\caption{The cross-section of velocity for different Wi and Co.\label{cross-section-velocity}}
\end{figure}

\begin{figure}[htbp]
\setlength{\abovecaptionskip}{1pt}
\setlength{\abovecaptionskip}{1pt}
\setlength{\belowcaptionskip}{1pt}
\renewcommand*{\figurename}{Fig.}
\centering
\begin{minipage}[t]{0.3\linewidth}
\includegraphics[width=1.80in, height=1.4in]{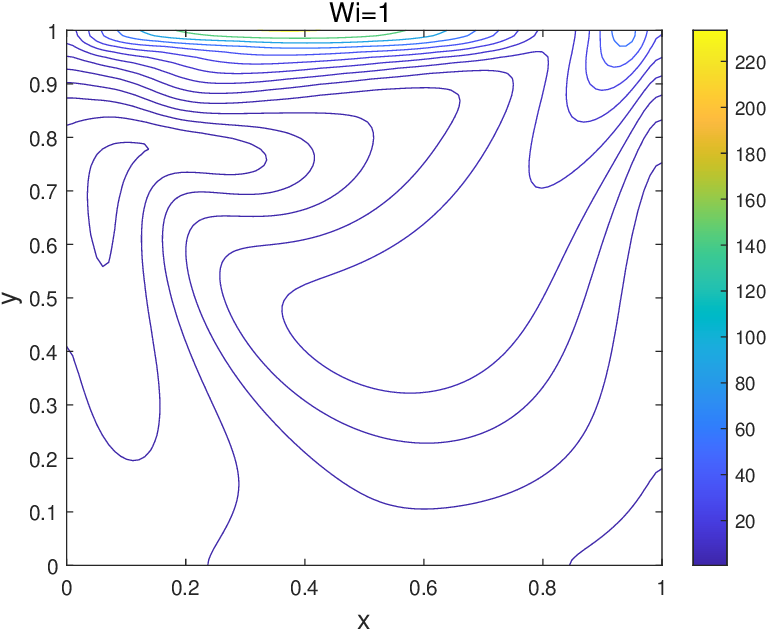}
\end{minipage}
\begin{minipage}[t]{0.3\linewidth}
\includegraphics[width=1.80in, height=1.4in]{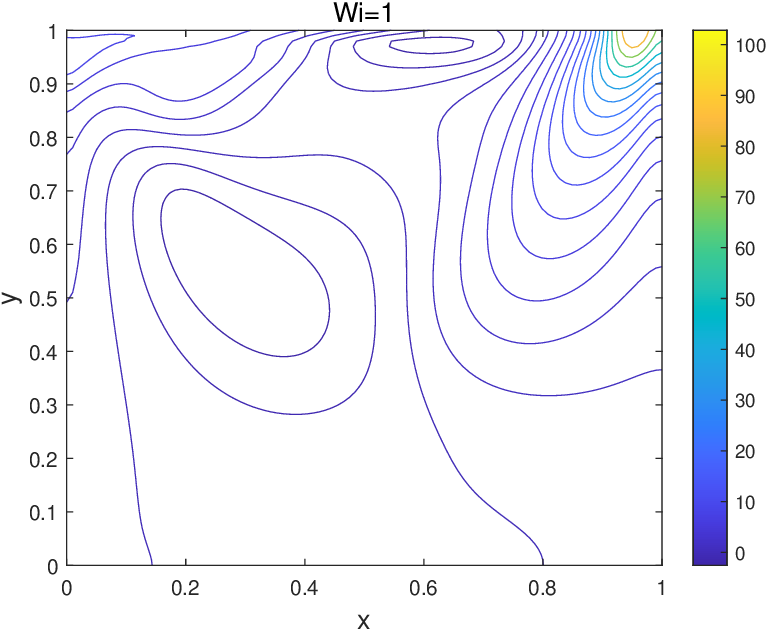}
\end{minipage}
\begin{minipage}[t]{0.3\linewidth}
\includegraphics[width=1.80in, height=1.4in]{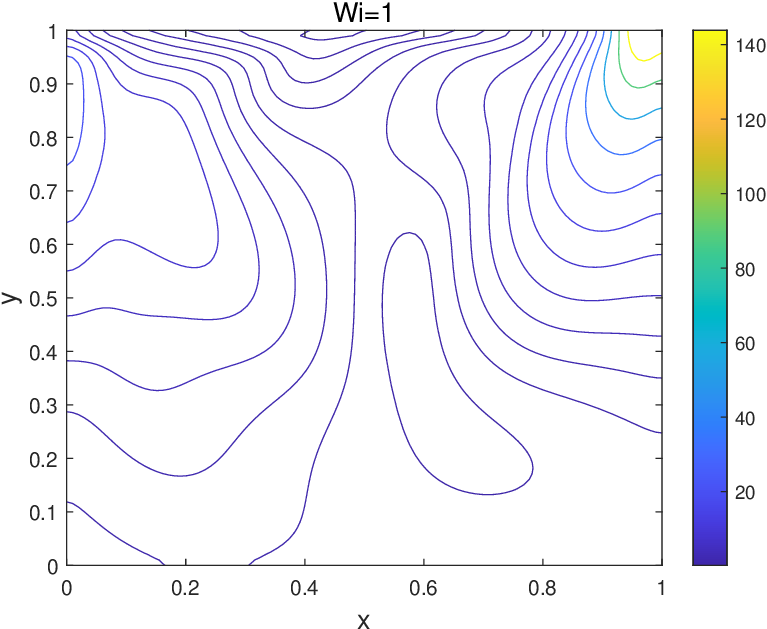}
\end{minipage}
\\
\begin{minipage}[t]{0.3\linewidth}
\includegraphics[width=1.80in, height=1.4in]{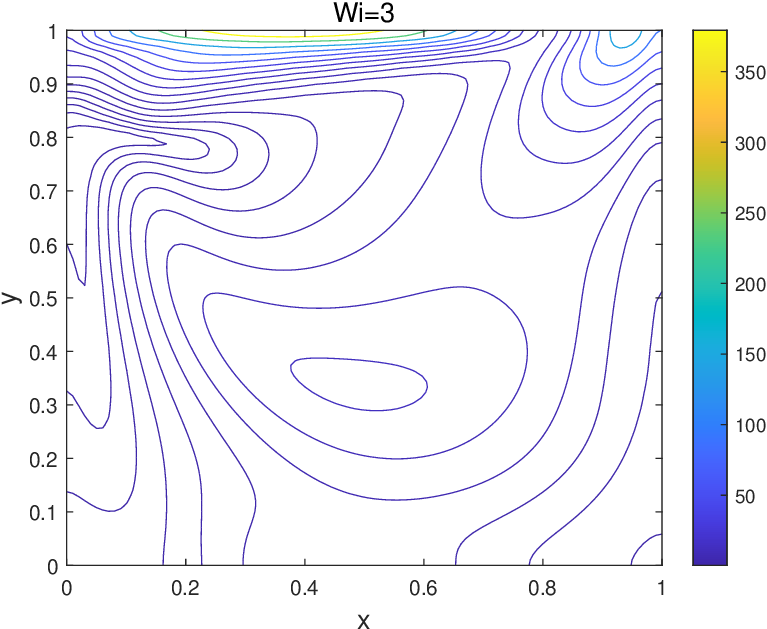}
\end{minipage}
\begin{minipage}[t]{0.3\linewidth}
\includegraphics[width=1.80in, height=1.4in]{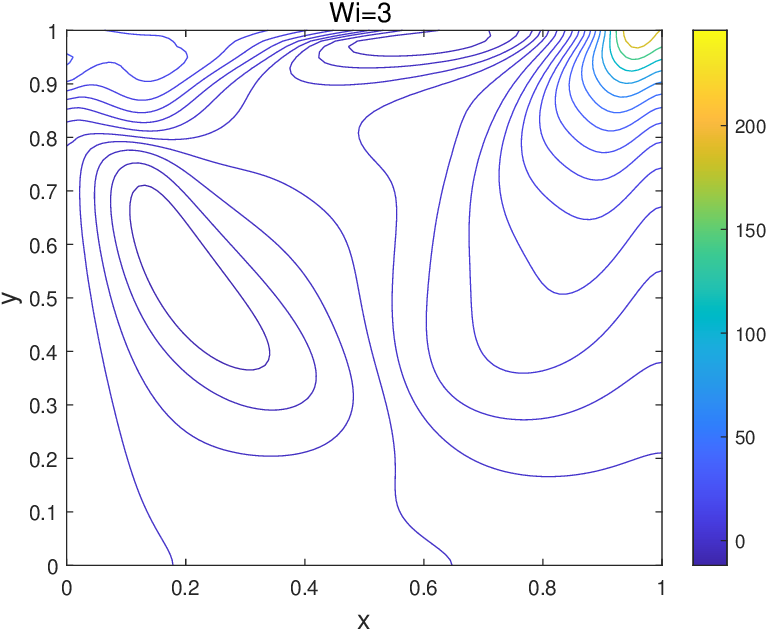}
\end{minipage}
\begin{minipage}[t]{0.3\linewidth}
\includegraphics[width=1.80in, height=1.4in]{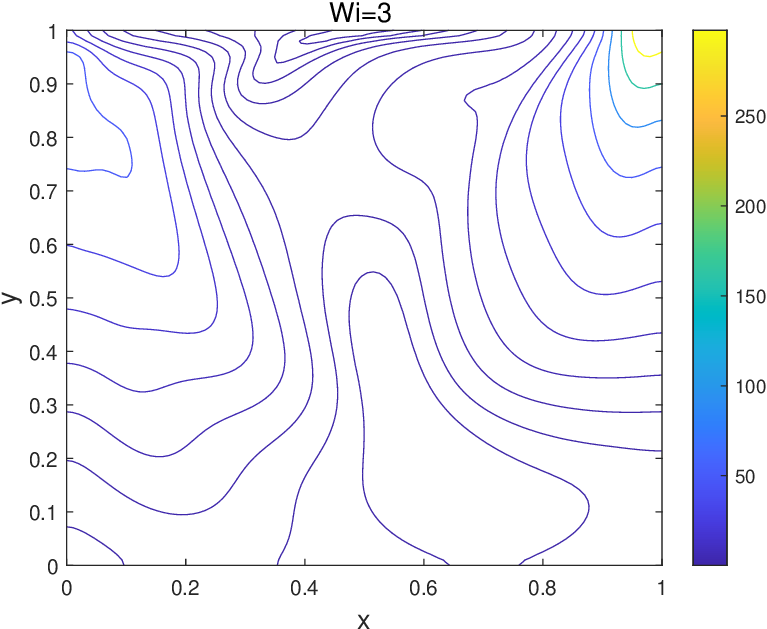}
\end{minipage}
\\
\begin{minipage}[t]{0.3\linewidth}
\includegraphics[width=1.80in, height=1.4in]{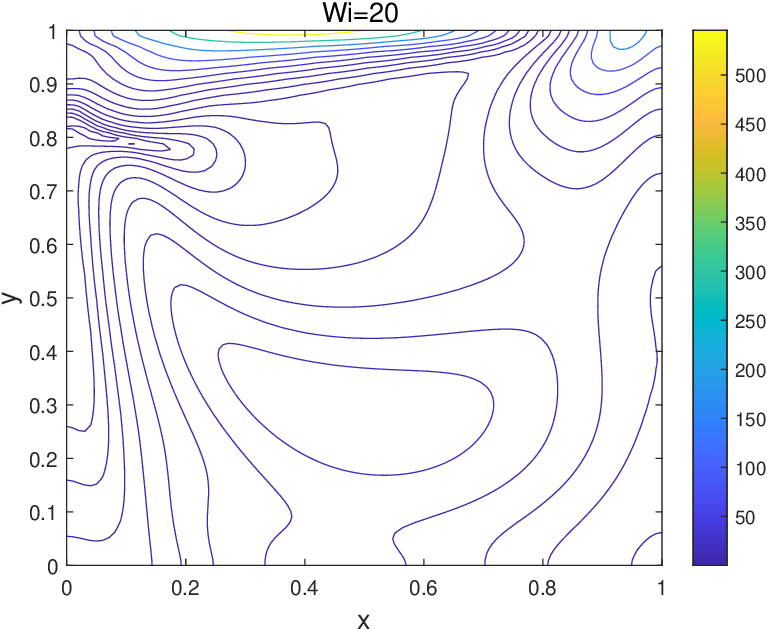}
\end{minipage}
\begin{minipage}[t]{0.3\linewidth}
\includegraphics[width=1.80in, height=1.4in]{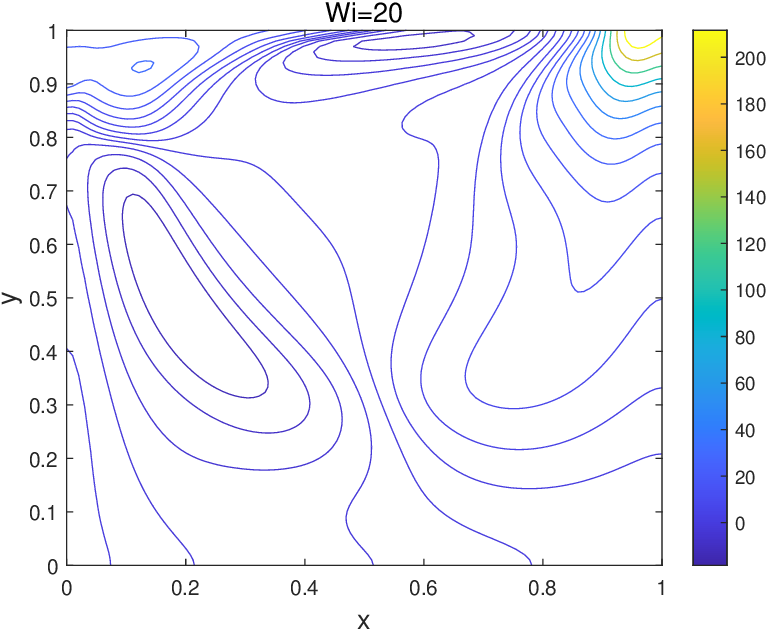}
\end{minipage}
\begin{minipage}[t]{0.3\linewidth}
\includegraphics[width=1.80in, height=1.4in]{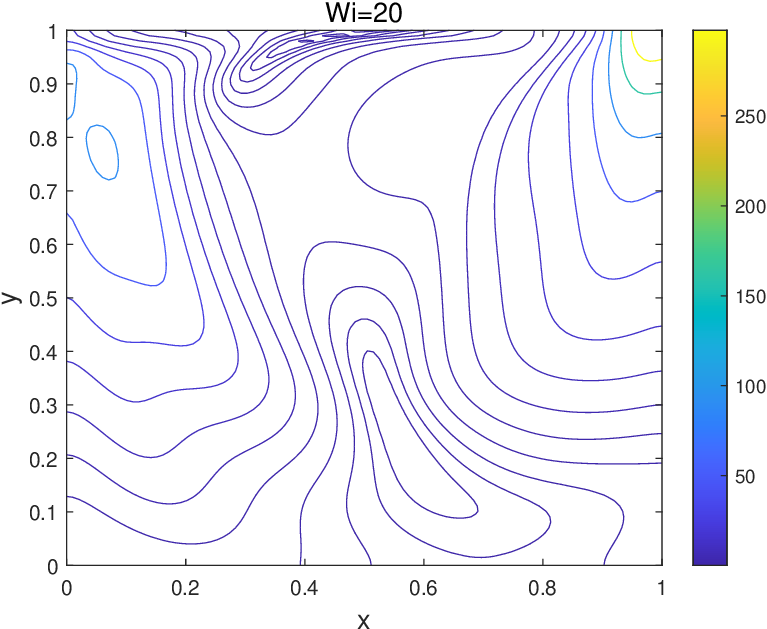}
\end{minipage}
\caption{The contour maps of the $\sigma_{11}$, $\sigma_{12}$ and $\sigma_{22}$ (left to right) for $Wi=1,3,20$.\label{Contour-tensor}}
\end{figure}

\begin{figure}[htbp]
\setlength{\abovecaptionskip}{1pt}
\setlength{\abovecaptionskip}{1pt}
\setlength{\belowcaptionskip}{1pt}
\renewcommand*{\figurename}{Fig.}
\centering
\begin{minipage}[t]{0.29\linewidth}
\includegraphics[width=1.8in, height=1.4in]{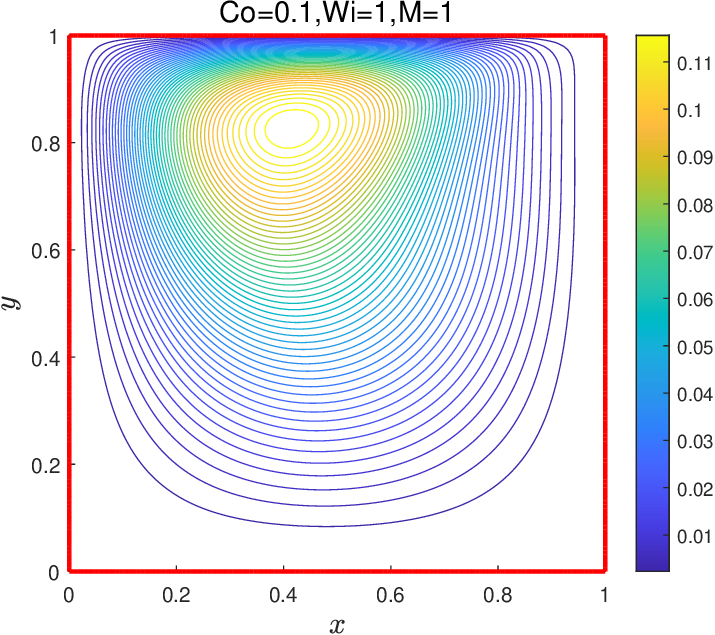}
\end{minipage}
\begin{minipage}[t]{0.29\linewidth}
\includegraphics[width=1.8in, height=1.4in]{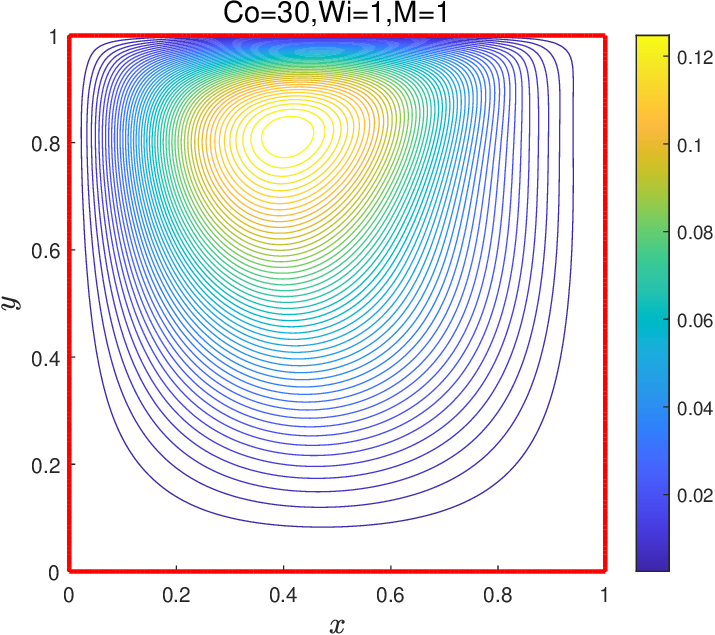}
\end{minipage}
\begin{minipage}[t]{0.29\linewidth}
\includegraphics[width=1.8in, height=1.4in]{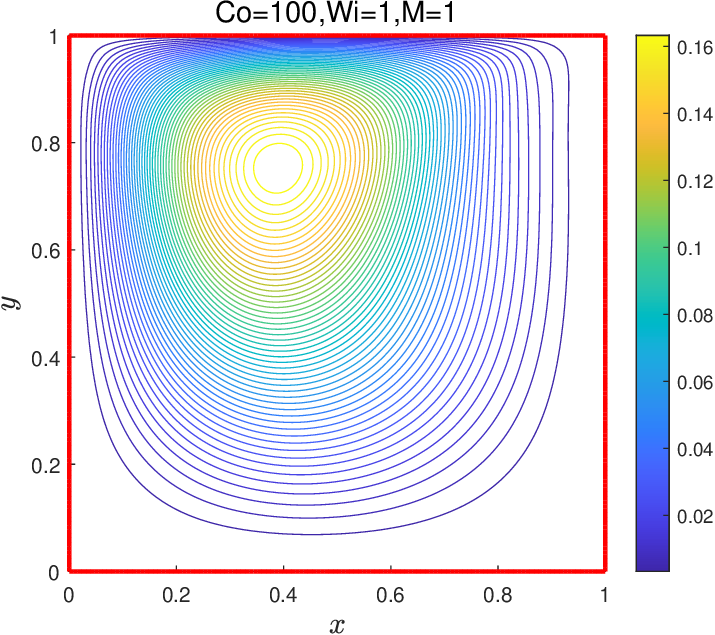}
\end{minipage}
\\
\begin{minipage}[t]{0.29\linewidth}
\includegraphics[width=1.8in, height=1.4in]{flowCo01Wi1EI1}
\end{minipage}
\begin{minipage}[t]{0.29\linewidth}
\includegraphics[width=1.8in, height=1.4in]{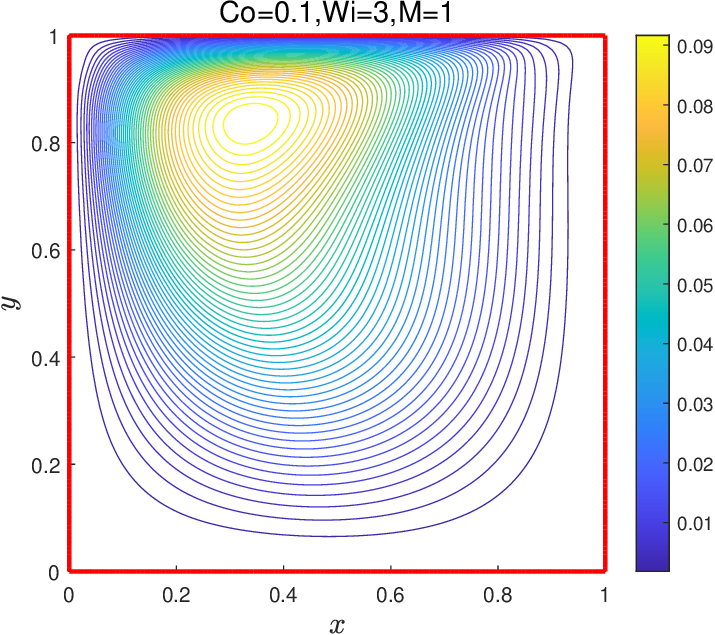}
\end{minipage}
\begin{minipage}[t]{0.29\linewidth}
\includegraphics[width=1.8in, height=1.4in]{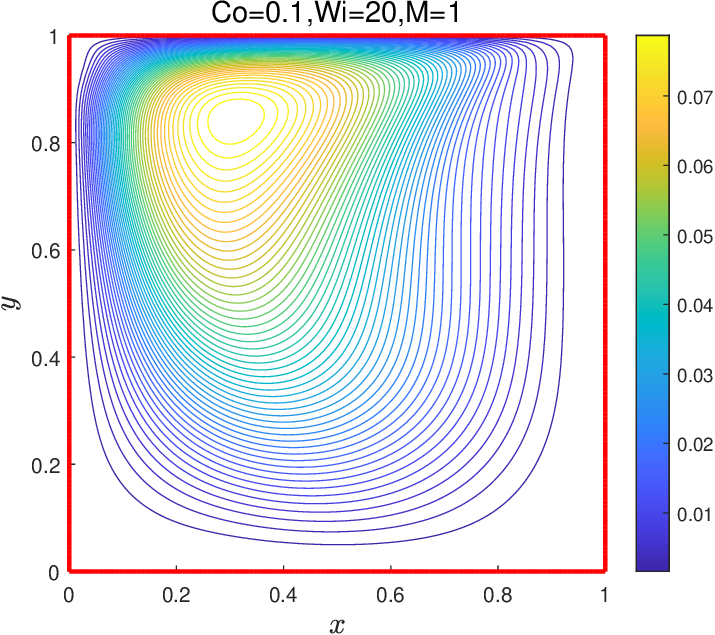}
\end{minipage}
\\
\begin{minipage}[t]{0.29\linewidth}
\includegraphics[width=1.8in, height=1.4in]{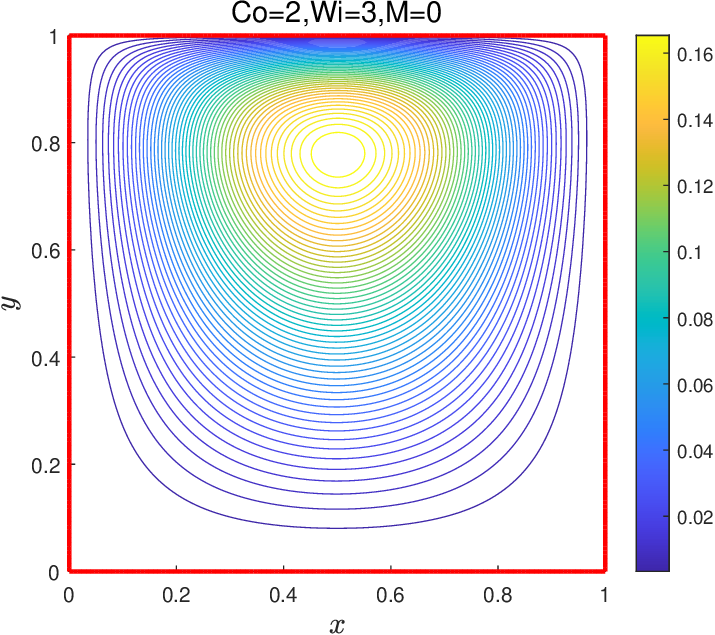}
\end{minipage}
\begin{minipage}[t]{0.29\linewidth}
\includegraphics[width=1.8in, height=1.4in]{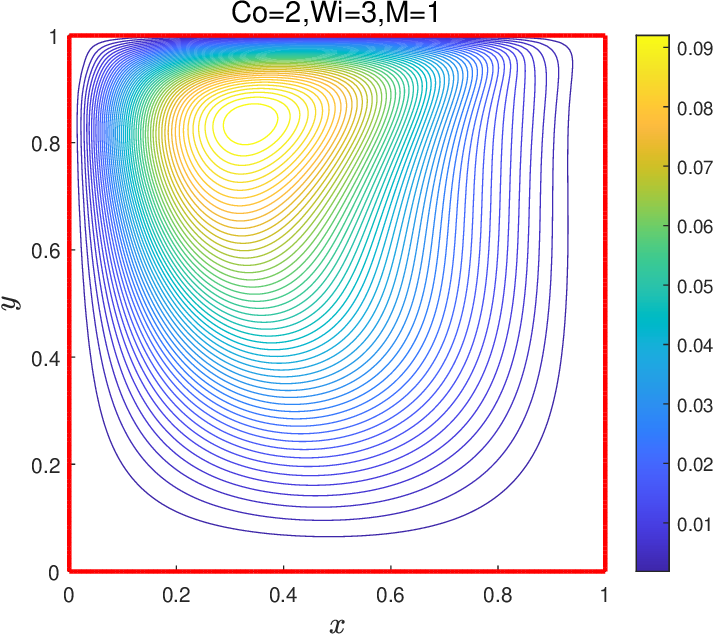}
\end{minipage}
\begin{minipage}[t]{0.29\linewidth}
\includegraphics[width=1.8in, height=1.4in]{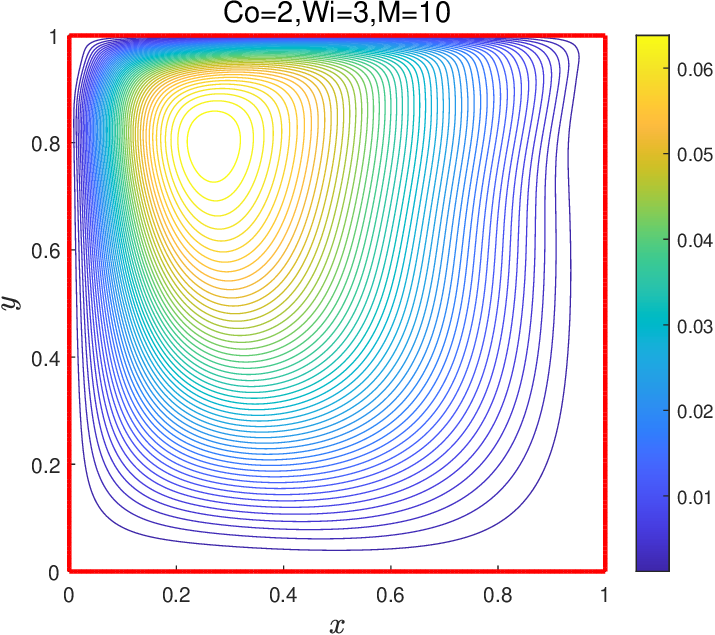}
\end{minipage}
\caption{ Streamline with different $Co$ number, $Wi$ number and $M$ number.\label{Lid-driven-Streamlines}}
\end{figure}

\section{Conclusions}
In this paper, we have developed and analyzed the first-order scheme for the Oldroyd-B electrohydrodynamic model, which is a complex system that describes the behavior of diffusive viscoelastic fluids under the influence of electric fields.
The proposed scheme is constructed based on an auxiliary variable approach for the flow equations and a splitting technique of the coupling terms.
The designed schemes has been rigorously proven to be energy stable, preserve positivity and mass conservation of the ionic concentrations. Moreover, they maintains the positive-definite property of the conformation tensor by logarithmic transformation. At each step, only linear and decoupled equations need to be solved, which significantly simplifies the computational complexity of the numerical implementation.
The numerical tests have confirmed the desired accuracy of the schemes and the theoretical claims.
In our numerical tests, the proposed scheme is efficient for large values of the Weissenberg number and it has been observed that the elastic effect of fluids influences the flow structures.


\bibliographystyle{unsrt}
\bibliography{reference}
	
\end{document}